\documentclass[a4paper,12pt,reqno]{amsart}
\usepackage{amsfonts}
\usepackage{amsmath}
\usepackage{amssymb}
\usepackage[a4paper]{geometry}
\usepackage{mathrsfs}
\usepackage[colorlinks]{hyperref}
\usepackage{comment}
\renewcommand\eqref[1]{(\ref{#1})} 
\numberwithin{equation}{section}
\theoremstyle{plain}
\newtheorem{thm}{Theorem}[section]
\newtheorem{proposition}[thm]{Proposition}
\newtheorem{cor}[thm]{Corollary}
\newtheorem{lemma}[thm]{Lemma}

\theoremstyle{definition}
\newtheorem{defn}[thm]{Definition}
\newtheorem{rem}[thm]{Remark}
\newtheorem{ex}[thm]{Example}

\newcommand{\Rn}{\mathbb R^{n}}
\newcommand{\Zn}{{\mathbb Z^{n}}}
\newcommand{\Tn}{{\mathbb T^{n}}}
\newcommand{\Op}{{\textrm{\rm Op}}}
\newcommand{\Optn}{{\textrm{\rm Op}_\Tn}}
\newcommand{\Opzn}{{\textrm{\rm Op}_{\hbar\Zn}}}
\newcommand{\HS}{{\mathtt{HS}}}

\def\Ftn{{\mathcal F}_\Tn}

\newcommand{\R}{\mathbb{R}}
\newcommand{\C}{\mathbb{C}}
\newcommand{\Z}{\mathbb{Z}}
\newcommand{\N}{\mathbb{N}}
\newcommand{\T}{\mathbb{T}}
\newcommand{\dd}{\,\mathrm{d}}
\newcommand{\mcF}{\mathcal{F}}

\begin{document}
	
   \title[Semi-classical pseudo-differential operators on $\hbar\mathbb{Z}^n$ and applications]
   {Semi-classical pseudo-differential operators on $\hbar\mathbb{Z}^n$ and applications}


\author[L. N. A. Botchway]{Linda N. A. Botchway}
\address{
  Linda N. A. Botchway:
  \endgraf
  University of Ghana
  \endgraf
  UG, Legon
  \endgraf
  Ghana
   \endgraf
  {\it E-mail address} {\rm lnabotchway001@st.ug.edu.gh}
  }

\author[M. Chatzakou]{Marianna Chatzakou}
\address{
	Marianna Chatzakou:
	\endgraf
	Department of Mathematics: Analysis, Logic and Discrete Mathematics
	\endgraf
	Ghent University, Belgium
	\endgraf
	Belgium
	\endgraf
	{\it E-mail address} {\rm marianna.chatzakou@ugent.be}
}
\author[M. Ruzhansky]{Michael Ruzhansky}
\address{
	Michael Ruzhansky:
	\endgraf
	Department of Mathematics: Analysis, Logic and Discrete Mathematics
	\endgraf
	Ghent University, Belgium
	\endgraf
	and
	\endgraf
	School of Mathematical Sciences
	\endgraf Queen Mary University of London 
	\endgraf
	United Kingdom
	\endgraf
	{\it E-mail address} {\rm Michael.Ruzhansky@ugent.be}  }

\thanks{LNA. Botchway is supported by the Carnegie Corporation of New York ( Banga-Africa project at the University of Ghana). M. Ruzhansky is partially supported by the FWO Odysseus 1 grant G.0H94.18N: Analysis and Partial Differential Equations, the Methusalem programme of the Ghent University Special Research Fund (BOF) (Grant number 01M01021) and by by EPSRC grant EP/R003025/2. M. Chatzakou is a postdoctoral fellow of the Research Foundation – Flanders (FWO)  under the postdoctoral grant No 12B1223N. LNA Bothcway is grateful to Prof. Benoit F. Sehba for the supervision of her PhD thesis and for mathematical discussions.}


     \begin{abstract}
     In this paper we consider the semiclassical version of pseudo-differential operators on the lattice space $\hbar \Zn$.
     The current work is an extension of the previous work \cite{BKR20} and agrees with it in the limit of the parameter $\hbar \rightarrow 1$. The various representations of the operators  will be studied as well as the  composition,  transpose, adjoint  and  the link between ellipticity and  parametrix  of operators. We also give the  conditions for the $\ell^p$, weighted $\ell^2$  boundedness  and $\ell^p$ compactness of operators. We investigate  the relation between the classical and semi-classical  quantization  in the spirit of \cite{ruzhansky2009pseudo} and \cite{RT-JFAA} and employ its applications to   Schatten-Von Neumann classes on $\ell^2( \hbar \mathbb{Z}^n)$.  We establish G{\aa}rding  and sharp G{\aa}rding inequalities, with an application to the well-posedness of parabolic equations on  the lattice $\hbar \mathbb{Z}^n$. Finally we verify that in the limiting case where $\hbar \rightarrow 0$ the semi-classical calculus of pseudo-differential operators recovers the classical Euclidean calculus, but with a twist.
     
     \end{abstract}
     \maketitle

\tableofcontents

\section{Introduction}

The main aim of this work is to develop a calculus of pseudo-differential operators  on the lattices 
\[
 \hbar \mathbb{Z}^n= \{ x \in \mathbb{R}^n : x=\hbar k\,,\quad k \in \mathbb{Z}^n \}\,,
\]
where $\hbar \in (0,1]$ is a small parameter. The particular case when $\hbar=1$ has been considered in \cite{BKR20}. Our analysis will allow to solve difference equations on the lattice $\hbar \mathbb{Z}^n$ that can appear either as the discretisation of the continuous counterpart, or naturally in modelling problems such as the visualization of physical phenomena etc. We also investigate the behaviour of the calculus in the limit $\hbar \rightarrow 0$.

The recent work \cite{CDRT23} on a semi-classical version of the nonhomogeneous heat equation on $\hbar \Z^n$ is one of the main motivations for our analysis. More generally, the analysis of Cauchy problems with space variable on the lattice, see e.g. the parabolic Anderson model \cite{Kon16} with many applications in real word problems, stimulate the current work and in particular the development of the semi-classical calculus. Importantly, and still referring to the example of the parabolic Anderson model, the investigation of the limiting case when $\hbar \rightarrow 0$, allows for the study of the continuous analog of the parabolic Anderson model \cite{GK00} initiating from its discretized analysis. 

Let us consider a rigorous and rather simple example of a discretised difference equation: for $g$ being a function on the lattice $\hbar \mathbb{Z}^n$, and $a \in \mathbb{C}$, we regard the equation:   
  \begin{equation}\label{EQ:ex0}
  \sum_{j=1}^{n}\Big(f(k + \hbar v_j) + f(k - \hbar v_j)\Big) -2 af(k)=g(k), \quad k\in\hbar \Zn,
  \end{equation}
where $v_j=(0,\cdots,0,1,0,\cdots) \in \mathbb{Z}^n$, and the only non-zero element of the vector $v_j$ is the $j^{\text{\tiny th}}$ element and is equal to $1$. In the case where $\text{Re}(a) \neq 0$, and $g \in \ell^2(\hbar \Zn)$ the solution $f$ to the equation \eqref{EQ:ex0} is given by the following expression 
\begin{equation}\label{EQ:ex1}
f(k) = \int_{\mathbb{T}^n}e^{2\pi \frac{i}{\hbar}k\cdot \theta}\frac{1}{2\sum_{j=1}^n \cos(2\pi \theta_j) + a} \widehat{g}(\theta)\text{d}\theta,
\end{equation} 
    where $\widehat{g}$ denotes the Fourier transform of $g$ on the lattice $\hbar \mathbb{Z}^n$; that is the expression
      \begin{equation}\label{gael}
    \widehat{g}(\theta) = \sum_{k\in \hbar\mathbb{Z}^{n}} e^{-2\pi \frac{i}{\hbar}k\cdot \theta} g(k),\quad \theta\in\Tn\,.
      \end{equation} 
      Additionally, it follows that whenever $g \in \ell^2(\hbar \Zn)$ then we also have $f \in \ell^2(\hbar \Zn)$. Importantly, we know that the formula \eqref{EQ:ex1} for the solution to the equation \eqref{EQ:ex0} can be extended to give solutions also in the case where $g$ is any tempered growth distribution, i.e., when $g\in\mathcal{S}'(\hbar \Zn)$. For instance, whenever $g$ satisfies the estimate
       $$
   \sum_{k\in\hbar \Zn} (1+|k|)^{s} |g(k)|^2<\infty
 $$
 then the solution $f$ satisfies the same estimate; that is we have $$
   \sum_{k\in\hbar \Zn} (1+|k|)^{s} |f(k)|^2<\infty\,,
 $$
see Example $(3)$ in Section \ref{SEC:EX}.

\smallskip We point out that the operators that are of the form \eqref{EQ:ex0} extend the classical notion of difference operators on a discrete setting like in particular the lattice case $\hbar \mathbb{Z}^n$. Indeed, as shown in Section \ref{Sec.h=0} the calculus of pseudo-differential operators in our semi-classical setting agrees with the classical pseudo-differential calculus in the Euclidean setting. Thus, in this work we adopt the terminology {\it pseudo-differential operators} as it exists already in the literature, see e.g. \cite{Rab09}, to describe the operators that we consider, emphasising this way that they extend the usual class of difference operators into a $*$-algebra.

\smallskip Hence, the analysis here aims to develop a global calculus of pseudo-differential operators on the lattice $\hbar \mathbb{Z}^n$ that will be employed to deal with problems around

\begin{itemize}
    \item the type of the difference equations that can be solved within the developed framework;
    \item the properties of the function $g$ as in \eqref{EQ:ex0} that can be transferred to the solution $f$;
    \item the solvability of equation of the form \eqref{EQ:ex0} in the case where the coefficient of the operators depend also on the variable $k \in \hbar \mathbb{Z}^n$.
\end{itemize}
\medskip

\section{Preliminary notions and tools}\label{Sec.pre}
In this section we aim to recall the necessary toolkit and notions that shall be used for our analysis developed in later sections.

\smallskip To begin with let us start with the formal definition of the Fourier transform of a function $f \in \ell^1(\hbar \mathbb{Z}^n)$ (semi-classical Fourier transform) that is given by 
\begin{equation} \label{gael1}
	\mcF_{\hbar\Z^n} f(\theta) := \hat{f}(\theta) := \sum_{k\in \hbar\Z^n} e^{-2\pi \frac{i}{\hbar}k \cdot \theta} f(k), \quad \theta\in \T^n=\mathbb{R}^n/\mathbb{Z}^n.
\end{equation}
In the formula \eqref{gael1}, as well as in the sequel, the product $k \cdot \theta$ for $k=\hbar(k_1,\cdots,k_n) \in \hbar \mathbb{Z}^n\,, \theta=(\theta_1,\cdots,\theta_m) \in \T^n$, is calculated by the following expression 
$$k\cdot \theta = \hbar\sum_{j=1}^{n}k_{j}\theta_{j}\,.$$
The Plancherel formula for the lattice $\hbar \mathbb{Z}^n$ reads
\begin{equation}\label{perrin}
	\sum_{k\in \hbar \mathbb{Z}^{n}}|f(k)|^{2} = \int_{\mathbb{T}^n}|\widehat{f}(\theta)|^{2}\text{d}\theta\,.
\end{equation}
To prove the inverse Fourier transform we perform the following computations: Let $f \in \ell^{1}(\hbar \Zn)$. If  we set $k=\hbar l$, $l \in \Zn$, and $f_{\hbar}(l):=f(\hbar l)$, then $f_{\hbar}$ is a function from $\Zn$ to $\mathbb{C}$ and using \eqref{gael1} we have:
\[
\mathcal{F}_{\hbar \Zn}f(\theta)=\sum_{k \in \hbar \Zn}e^{-2\pi \frac{i}{\hbar}k\cdot \theta}f(k)=\sum_{l \in \Zn}e^{-2\pi i l \cdot \theta}f(\hbar l)=\mathcal{F}_{\Zn}f_{\hbar}(\theta)\,.
\]
Now, using the inverse Fourier transform on the lattice $\Zn$ we can write 
\[
f_{\hbar}(l)=\int_{\Tn}e^{2\pi i l \cdot \theta}\mathcal{F}_{\Zn}f_{\hbar}(\theta)\text{d}\theta=\int_{\Tn}e^{2\pi \frac{i}{\hbar}k \cdot \theta} \widehat{f}(\theta)\text{d}\theta\,,
\]
and thus we  have shown that the inverse Fourier transform on $\hbar \Zn$  is given by 
\begin{equation}\label{EQ:Finv}
	f(k) = \int_{\mathbb{T}^n}e^{2\pi  \frac{i}{\hbar}k\cdot \theta}\widehat{f}(\theta)\text{d}\theta, \ \ \ k\in \mathbb{Z}^n\,.
\end{equation}

A measurable function $\sigma_\hbar:\hbar\Z^n \times \T^n \to \C$, defines a sequence $\Op_{\hbar}(\sigma_\hbar)$ by
\begin{equation}\label{neil}
	\Op_{\hbar}(\sigma_\hbar)f(k) : = \int_{\T^n} e^{2\pi \frac{i}{\hbar}k\cdot \theta} \sigma_\hbar(k,\theta) \mcF_{\hbar\Z^n} f(\theta)\dd \theta\,,
\end{equation}
provided that there are some rational restrictions on $\sigma_{\hbar}$. The operator \eqref{neil} shall be called a semi-classical pseudo-differential operator on $\hbar \mathbb{Z}^n$  corresponding to the \textit{symbol} $\sigma_{\hbar}(k,\theta)$ on $\hbar \mathbb{Z}^n \times \T^n$, or in short a $\Psi_{\hbar}DO$. The process of associating a symbol $\sigma_{\hbar}$ to a pseudo-differrential operator $\Op_{\hbar}(\sigma_\hbar)$, i.e., the mapping $\sigma_{\hbar} \mapsto \Op_{\hbar}(\sigma_\hbar)$, is called $\hbar \mathbb{Z}^n$-quantization, or simply \textit{quantization}.

\smallskip

The space of rapidly decreasing functions $\mathcal{S}(\hbar \mathbb{Z}^n)$ on the lattice shall be called the \textit{Schwartz space}. This consists of functions $\varphi : \hbar \mathbb{Z}^n \rightarrow \T^n$ for which for every $N < \infty$ there exists a constant $c_{\varphi,N}$ (depending on the function $\varphi$ and on the choice of $N$) so that 
\[
|\varphi (k) | \leq c_{\varphi,N} (1+|k|)^{-N}\,, \quad \text{for all} \quad k \in \hbar \mathbb{Z}^n\,,
\]
where we have denoted by $|k|$ the $\ell^2$-norm of $k$, i.e., we have $|k|=\hbar \left(\sum_{j=1}^{n} k_{j}^{2}\right)^{\frac{1}{2}}$. The topology of $\mathcal{S}(\hbar \mathbb{Z}^n)$  is given by the seminorms $p_j(\varphi)$, $j \in \mathbb{N}_0$ \footnote{Throughout the paper we will use the notation $\mathbb{N}_0=\mathbb{N}\cup\{0\}.$}, where $p_j(\varphi):= \sup_{k \in \hbar \mathbb{Z}^n} (1+|k|)^j |\varphi(\xi)|$. The space of \textit{tempered distributions} $\mathcal{S}'(\hbar \mathbb{Z}^n)$ is the dual of $\mathcal{S}(\hbar \mathbb{Z}^n)$; that is the continuous linear functionals on $\mathcal{S}(\hbar \mathbb{Z}^n)$.

\smallskip

\par \textbf{Informal discussion.} \quad The main underlying idea behind the definition of the pseudo-differential operator as \eqref{neil} is that given a linear continuous  operator $A: \ell^{\infty}(\hbar \mathbb{Z}^n) \rightarrow \mathcal{S}'(\hbar\mathbb{Z}^n)$,  the image of the functions $e_\theta = (k \mapsto e^{2\pi \frac{i}{\hbar}k \cdot \theta})$ for $\theta \in \T^n$ via the operator $A$ completely determines the operator $A$. To this end we define the symbol $\sigma_{\hbar}$ of operator $A=\Op(\sigma_\hbar)$ by testing the operator $A$ on the functions $e_\theta$ yielding $Ae_{\theta}(k)=e^{2\pi \frac{i}{\hbar} k \cdot \theta} \sigma_{\hbar}(k, \theta)$, i.e., we define 
\begin{equation}
    \label{def.symb}
    \sigma_{\hbar}(k,\theta):= e^{- 2 \pi \frac{i}{\hbar} k \cdot \theta} Ae_{\theta}(k)\,,
\end{equation}
see Proposition \ref{PROP:symbols} for the proof of \eqref{def.symb}.

\smallskip 
We claim  that for a symbol $\sigma_{\hbar}$ as in \eqref{def.symb} the operator $A$ is indeed the operator arising as the quantization of $\sigma_{\hbar}$. Indeed, with the use of the inverse Fourier transform \eqref{EQ:Finv} we have 
\begin{eqnarray*}
    A f(k) & = & A \left( \int_{\T^n}e^{2\pi \frac{i}{\hbar}k \cdot \theta}\widehat{f}(\theta)\,\text{d}\theta \right) \\
   &  = &  \int_{\T^n} A \left( e^{2\pi \frac{i}{\hbar}k \cdot \theta}\right)\widehat{f}(\theta)\,\text{d}\theta \\
   & = & \int_{\T^n} e^{2 \pi \frac{i}{\hbar} k \cdot \theta}\sigma_{\hbar}(k,\theta)\widehat{f}(\theta)\,\text{d}\theta = \Op(\sigma_\hbar)f(k)\,,
\end{eqnarray*}
and we have proved our claim.

 \section{Representation of $\Psi_\hbar$DO's}
\label{SEC:symbols}
\subsection{Symbol classes}
   Let us begin this section by defining the notion of difference operators (or semi-classical difference operators) in our setting; these are exactly the operators that can be served as the analogues of the derivatives with respect to the Fourier variable in the Euclidean setting.

    \begin{defn}[Semi-Classical Difference Operator]
    \label{DEF:diff.op}
    	For  $\alpha = (\alpha_1, \dots, \alpha_n)$, we define the difference operator $ \Delta^{\alpha}_{\hbar}$ in our setting, as the operator acting on functions $ g : \hbar \Z^n \rightarrow \C$ via
    	
   \begin{equation}\label{EQ:diffs}
   		\Delta^{\alpha}_{\hbar} g(k) 	= \frac{1}{\hbar^{|\alpha|}} \int_{\T^n}  e^{2\pi \frac{i}{\hbar} k\cdot \theta} \Big(  e^{2\pi \frac{i}{\hbar} \theta } -1   \Big)^{\alpha} \widehat{g} (\theta) d\theta\,,
   \end{equation}
where we have used the notation
\begin{equation} 
\label{e.exp}
\Big(  e^{2\pi\frac{i}{\hbar} \theta } -1   \Big)^{\alpha} = \Big(  e^{2\pi\frac{i}{\hbar}\theta_1 } -1   \Big)^{\alpha_1} \dots \Big(  e^{2\pi\frac{i}{\hbar}\theta_n } -1   \Big)^{\alpha_n}\,.  \end{equation}
The usual (semi-classical) difference operators $\Delta_{\hbar,j}$, $j=1,\cdots,n$, on $\hbar \mathbb{Z}^n$ are as follows: Let $v_j=(0,\dots,0,1,0,\dots,0)$ be the vector with $1$ is at the $j^{th}$ position. Then the formula for the operator $\Delta_{\hbar,j}$ when acting on $g$ is given by
\begin{eqnarray}\label{EQ:diffs2}
\Delta_{\hbar,j} g(k) &=& \frac{1}{\hbar}\left[  \int_{\T^n} e^{2\pi \frac{i}{\hbar} (k+\hbar v_j )\cdot \theta} \widehat{g} (\theta) d\theta - \int_{\T^n} e^{2\pi \frac{i}{\hbar} k\cdot\theta} \widehat{g} (\theta) d\theta  \right] \\ 
 &=& \frac{g(k+\hbar v_j) - g(k)}{\hbar}\,.
\end{eqnarray}
It is then easy to check  the following  decomposition 
\begin{equation}\label{EQ:diffs0} 
	\Delta^{\alpha}_{\hbar} = \Delta^{\alpha_1}_{\hbar,1} \cdot \dots \cdot \Delta^{\alpha_n}_{\hbar,n}\,.
\end{equation}

    \end{defn}
\begin{rem}
\begin{itemize}
    \item We note that formulae \eqref{EQ:diffs0} and \eqref{EQ:diffs2} give an alternative characterisation to representation \eqref{EQ:diffs}. Hence, their combination can be considered instead as the definition of the (semi-classical) difference operators $\Delta_{\hbar \Z^n}$.
    \item It is easy to verify that the difference operators satisfy many useful properties, including the Leibniz formula, summation by parts formula, and Taylor expansion formula; see \cite{RT-JFAA} and
    \cite[Section 3.3]{ruzhansky2009pseudo}. 
\end{itemize}

\end{rem}
    We point out that the representation formula \eqref{EQ:diffs} is applicable also to $g\in{\mathcal S}'(\hbar \Zn)$. Indeed, in this case we have $\widehat{g}\in\mathcal{D}'(\Tn)$ and the formula \eqref{EQ:diffs} can be viewed in terms of the distributional duality on $\Tn$; i.e., it reads as follows
    \begin{equation}\label{EQ:ddif}
 \Delta^{\alpha}_{\hbar} g (k)=\frac{1}{\hbar^{|\alpha|}}\langle \widehat{g},e^{2\pi \frac{i}{\hbar} k\cdot \theta} (e^{2\pi i  \theta} -1 )^\alpha\rangle\,.
\end{equation} 
   The following operators shall be used in the definition of symbol classes. Additionally, they are useful in the torodial analysis, and their precise form, see  \eqref{def.PD} is related to the Stirling numbers; see  \cite[Section 3.4]{ruzhansky2009pseudo} for a detailed discussion.
    \begin{defn}[Partial derivatives on $\T^n$] \label {part.der.theta}
   For our purposes it is useful to introduce the partial derivatives type operators on $\T^n$ as follows. For $\beta \in \N^{n}_{0}$ we define:

   \begin{equation*}
\begin{split}
D_{\hbar, \theta}^{(\beta)} & := D_{\hbar, \theta_1}^{(\beta_1)}\cdots D_{\hbar, \theta_n}^{(\beta_n)}\,,\\
D_{\hbar, \theta}^{\beta} & : = D_{\hbar, \theta_1}^{\beta_1}\cdots D_{\hbar, \theta_n}^{\beta_n}\,,
\end{split}
\end{equation*}
where for $\beta_j \in \N_{0}$,
\begin{equation}
    \label{def.PD}
    \begin{split}
    D_{\hbar, \theta_j}^{(\beta_j)} & :=\hbar^{\beta_j}\left(\prod_{\ell=0}^{\beta_j-1}\frac{1}{2\pi i}\frac{\partial}{\partial \theta_j}-\ell\right)\,, \\
    D_{\hbar, \theta_j}^{\beta_j} & := \hbar^{\beta_j}\left(\frac{1}{2\pi i}\frac{\partial}{\partial \theta_j} \right)^{\beta_j}\,.
    \end{split}
\end{equation}
By the above it follows that 
\[
D_{\hbar, \theta}^{(\beta)} =\hbar^{|\beta|}\left(\prod_{\ell=0}^{\beta_1-1}\frac{1}{2\pi i}\frac{\partial}{\partial \theta_1}-\ell\right)\cdots \left(\prod_{\ell=0}^{\beta_n-1}\frac{1}{2\pi i}\frac{\partial}{\partial \theta_n}-\ell\right)\,.
\]
As usual, we denote $D_{\hbar, \theta}^{0}=D_{\hbar, \theta}^{(0)}=I$.
      
    \end{defn}
    We can then proceed to the definition of the classes of symbols that correspond to the $\hbar \Z^n$ quantization of operators.

    \begin{defn}[Symbol Classes $S_{\rho,\delta}^\mu(\hbar\Z^n \times\T^n)$] \label{DEF:symbolclasses}  
    	Let $\rho,\delta\in \R$. We say that a function $\sigma_\hbar: \hbar\Z^n \times\T^n\to \C$ is a \textit{symbol} that belongs to the (semi-classical) symbol class $S_{\rho,\delta}^\mu(\hbar\Z^n \times\T^n)$ if $\sigma_\hbar(k, \cdot)\in C^\infty(\T^n)$ for all 
    	$k\in \hbar\Z^n$, and for all multi-indices $\alpha,\beta \in \mathbb{N}_{0}^{n}$, there exists a positive constant $C_{\alpha,\beta}$ so that
    
     \begin{equation}
     		|D_{\hbar, \theta}^{(\beta)} \Delta^\alpha_{\hbar,k} \sigma_{\hbar}(k,\theta)|\le C_{\alpha,\beta}(1+|k|)^{\mu-\rho|\alpha|+\delta|\beta|},
    	\end{equation}
    	for all where $ k\in \hbar\Z^n\,, \theta \in\T^n$.
    	
    	If $\rho=1$ and $\delta=0$, we will denote simply $S^\mu(\hbar\Z^n \times\T^n):=S_{1,0}^\mu(\hbar\Z^n \times\T^n)$.
     As noted already in Section \ref{Sec.pre} we shall denote by $\Op_{\hbar}(\sigma_\hbar)$ the operator with symbol $\sigma_{\hbar}$ given by 

     \begin{equation*}\label{neil2}
	\Op_{\hbar}(\sigma_\hbar)f(k) : = \int_{\T^n} e^{2\pi \frac{i}{\hbar}k\cdot \theta} \sigma_\hbar(k,\theta) \mcF_{\hbar\Z^n} f(\theta)\dd \theta\,.
\end{equation*}
The family of (semi-classical) pseudo-differential operators with symbols in the class $S_{\rho,\delta}^\mu(\hbar\Z^n \times\T^n)$ will be denoted by $\Op_{\hbar}(S_{\rho,\delta}^\mu(\hbar\Z^n \times\T^n))$.
    \end{defn}
   
   We sometimes denote $\Delta^{\alpha}_{\hbar}=\Delta^\alpha_{\hbar,k}$ to underline the fact that these difference operators are acting  with respect to the variable lattice variable $k \in \hbar \mathbb{Z}^n$.
   \begin{rem}
       The symbol classes $S_{\rho,\delta}^\mu( \mathbb{T}^n \times  \mathbb{Z}^n )$ (that is, modulo interchanging the order of the lattice and toroidal variables $k$ and $\theta$)  have been extensively studied in \cite{RT-JFAA} in the toroidal setting $\T^n$. We also refer the interested reader to the monograph  \cite[Chapter 4]{ruzhansky2009pseudo} for a more thorough analysis of their properties.  Additionally, we note that the equivalence of the $S_{\rho,\delta}^\mu(\mathbb{T}^n \times \mathbb{Z}^n)$ classes, in the toroidal and compact Lie group case, to the usual H\"{o}rmander classes is proven in \cite{ruzhansky2014hormander}.
   \end{rem} 
   The \textit{smoothing} class of operators introduced below is related to the notion of invertibility of pseudo-differential operators that is discussed in Section \ref{SEC:calculus}.
   
  \begin{defn}[Smoothing semi-classical pseudo-diefferential operators]\label{smooth.semi}
      We say that a symbol $\sigma_\hbar$ is of order $-\infty$, and we write $\sigma_{\hbar} \in S^{-\infty}(\hbar \Z^n \times \T^n)$, if for all $(k,\theta) \in \hbar \Z^n \times \T^n$ we have
      \[
      |D_\theta^{(\beta)} \Delta^\alpha_{\hbar,k} \sigma(k,\theta)|\le C_{\alpha,\beta, N} (1+|k|)^{-N}\,,
      \]
      for all $N \in \N$. The latter condition is equivalent to writing that $\sigma_\hbar \in S^{\mu}_{1,0}(\hbar \Z^n \times \T^n)$  for all $\mu \in \R$. Formally we have \[S^{-\infty}(\hbar \mathbb{Z}^n \times \mathbb{T}^n):= \bigcap_{\mu \in \R} S^{\mu}_{1,0}(\hbar \mathbb{Z}^n \times \mathbb{T}^n)\,.\] The corresponding (semi-classical) pseudo-differential operators $\Op(\sigma_{\hbar})$ may be called \textit{smoothing pseudo-differential operators}.\footnote{In the lattice setting the terminology ``smoothing'' is used abusively; in general in discrete setting smoothing operators are  whose symbols that have rapid decay.}
  \end{defn}

    \subsection{Kernel of  $\Psi_\hbar$DO's} Using the Fourier transform \eqref{gael} we deduce an alternative representation of a semi-classical pseudo-differential representation; the so-called \textit{kernel representation}. 

  \smallskip
  For suitable functions $f$ we can write:
   \begin{eqnarray*}
   	\Op_{\hbar}(\sigma_\hbar)f(k) &=& \int_{\T^n} e^{2\pi \frac{i}{\hbar} k\cdot \theta} \sigma_\hbar(k,\theta) F_{\hbar\Z^n} f(\theta)\dd \theta \\
   		&=& \int_{\T^n}  \sum_{m\in \hbar\Z^n}e^{2\pi \frac{i}{\hbar} (k-m)\cdot \theta} \sigma_\hbar(k,\theta) f(m)\dd \theta \\
     &=& \sum_{m\in \hbar\Z^n}  \int_{\T^n} e^{2\pi \frac{i}{\hbar} (k-m)\cdot \theta} \sigma_\hbar(k,\theta) f(m)\dd \theta\\
   	&=& \sum_{m\in \hbar\Z^n} \int_{\T^n} e^{2\pi \frac{i}{\hbar} (k-m)\cdot \theta} \sigma_\hbar(k,\theta) f(m)\dd \theta \\
   	&=& \sum_{l\in \hbar\Z^n} \int_{\T^n} e^{2\pi \frac{i}{\hbar} l\cdot \theta} \sigma_\hbar (k,\theta) f(k-l)\dd \theta \\
   	&=& \sum_{l\in \hbar\Z^n} \kappa(k,l) f(k-l) \\
   	&=& \sum_{m\in \hbar\Z^n} K(k,m) f(m)\,.
   \end{eqnarray*} 

   Thus the kernel of $\Op_{\hbar}(\sigma_{\hbar})$ is given by
   \begin{equation}\label{EQ:kernels}
   	K(k,m)  = \kappa(k,k-m ) \quad \text{where} \quad
   	\kappa(k,l) = \int_{\T^n} e^{2\pi \frac{i}{\hbar} l\cdot \theta} \sigma_\hbar(k,\theta)\dd \theta.
   \end{equation}
   \smallskip
The next theorem establishes an important property of the kernel $K(k.m)$ of a semi-classical pseudo-differential operator in the class of operators $\Op(S_{\rho,\delta}^\mu)$.
 
 \begin{thm}\label{THM:kernel}
 	For $\delta\geq 0$, let $\sigma_{\hbar}\in S^{\mu}_{\rho, \delta}( \hbar \mathbb{Z}^n\times \mathbb{T}^n)$. Then, the kernel $K(k,m)$ of the pseudo-differential operator $\Op(\sigma_{\hbar})$ satisfies the following property 
 	
 	\begin{equation}\label{EQ:kernelsprops}
 		\Big|K\Big(k, m \Big)\Big|\leq C_{Q} \Big(1+\big|k\big|\Big)^{\mu + 2Q\delta} \Big(1+\frac{1}{\hbar}\Big| k-m\Big|\Big)^{-2Q}\,,\quad \forall k,m \in \hbar \Z^n\,, 
 	\end{equation}
 	for all $Q \in\mathbb{N}_0$, and for some  positive constant $C_{Q}>0$.
 \end{thm}
 \begin{rem}
     Before turning over to prove Theorem \ref{THM:kernel} let us note that, in contrast to the case of pseudo-differential operators on $\R^n$ or $\T^n$, the kernel $K(k,m)$ is well defined on the diagonal $k=m$ due to the discrete nature of the lattice $\hbar \Z^n \times \hbar \Z^n$.
 \end{rem}

 \begin{proof}[Proof of Theorem \ref{THM:kernel}] 
 	
 		  Let us first assume that $k=m$. Then, by definition of the kernel we have
 		
 		\begin{equation}\label{EQ:kerdiag}
 			K\Big(k,k\Big) =\kappa(k, 0) = \int_{\mathbb{T}^n}\sigma_{\hbar}(k, \theta)\text{d}\theta,
 		\end{equation} 
 		which immediately satisfies \eqref{EQ:kernelsprops} by the definition of the symbol class $S_{\rho,\delta}^\mu$. In the case where $k\not=m$, then also
   $l=k-m\not=0$. Let the Laplacian on the torus $\T^n$ be denoted by $\displaystyle \mathcal{L}_\theta$. Then straightforward computations give 
   \[
(1-\displaystyle \mathcal{L}_\theta)e^{2\pi \frac{i}{\hbar}l\cdot \theta}=\left(1-\sum_{j = 1}^{n}\frac{\partial ^2}{\partial \theta_j^2} \right)e^{2\pi \frac{i}{\hbar}l\cdot \theta}= \big(1 + \frac{4\pi ^{2}}{\hbar^2} \big|l \big|^{2}\big)e^{2\pi  \frac{i}{\hbar}l\cdot \theta} 
   \]
   which implies that
   \begin{equation}
   \label{int.by.parts}
e^{2\pi  \frac{i}{\hbar}l\cdot \theta} = \frac{(1 - \mathcal{L}_\theta)}{1 + \frac{4\pi ^{2}}{\hbar^2} \big|l \big|^{2}}  e^{2\pi  \frac{i}{\hbar}l\cdot \theta}\,.   
   \end{equation}
 	Substituting the above in the expression for $\kappa$ as in \eqref{EQ:kernels} we have
 	\begin{eqnarray*}
 		\kappa\Big(k, l\Big) &=& \int_{\mathbb{T}^n}e^{2\pi  \frac{i}{\hbar}l\cdot \theta}\sigma_{\hbar}(k, \theta)\text{d}\theta\\
 		&=& \int_{\mathbb{T}^n}\Bigg(\frac{(1 - \mathcal{L}_\theta)^Q}{\Big( 1 + \frac{4\pi ^{2}}{\hbar^2} \big|l\big|^{2}\Big)^Q}  e^{2\pi  \frac{i}{\hbar}l\cdot \theta}\Bigg)\sigma_{\hbar}(k, \theta)\text{d}\theta\\
 		&=& \Big(1 + \frac{4\pi ^{2}}{\hbar^2} \big|l\big|^{2}\Big)^{-Q}\int_{\mathbb{T}^n}e^{2\pi  \frac{i}{\hbar} l\cdot \theta}\big(1 - \mathcal{L}_\theta\big)^Q\sigma_{\hbar}(k, \theta)\text{d}\theta.
 	\end{eqnarray*}
 Hence, by the avove, taking the modulus of $\kappa(k,l)$ we have 
 \[
|\kappa(k,l)|\leq  \Big(1 + \frac{4\pi ^{2}}{\hbar^2} \big|l\big|^{2}\Big)^{-Q} |\big(1 - \mathcal{L}_\theta\big)^Q\sigma_{\hbar}(k, \theta)|\,,
 \]
 	which in turn by the assumption on the symbol $\sigma_{\hbar}$ gives
 	\[\big|\kappa\Big(k, l\Big)\big|\leq  C_{Q}  \Big(1+\big|k\big|\Big)^{\mu + 2Q\delta}\Big(1 + \frac{4\pi ^{2}}{\hbar^2} \big| l \big|^{2}\Big)^{-Q}\,,\]
  for all $Q \geq 0$. The latter gives the desired estimate if one takes into account \eqref{EQ:kernels}. The proof of Theorem \ref{THM:kernel} is  now complete.
 \end{proof}
 
   Similarly to the classical cases, one can extract the symbol of a given semi-classical pseudo-differential operator on $\hbar \Z^n$. The next result provides us with the corresponding formula.
   
   \begin{proposition}\label{PROP:symbols}
   	The symbol $\sigma_{\hbar}$ of a semi-classical pseudo-difference operator $T$ on $\hbar \Z^n$ 
   	is given by 
   	\begin{equation}\label{B}
   		\sigma_{\hbar}(k,\theta) = e^{-2\pi \frac{i}{\hbar} k\cdot \theta}Te_\theta(k),
   	\end{equation}
   	where $e_\theta (k) = e^{2\pi \frac{i}{\hbar} k\cdot \theta},$ for all $k\in\hbar\Z^n$ and for $\theta\in\T^n$.
   \end{proposition}

   \begin{proof}
   	For $\omega \in \T^n$ let $  e_\omega(l): = e^{2\pi \frac{i}{\hbar}  l\cdot \omega}$ where $l \in \hbar \Z^n$. Using \eqref{gael1}, the Fourier transform  of $ e_\omega$ is given  by  
   	\[
   	\widehat{  e_\omega} (\theta)  = \sum_{l \in \hbar \mathbb{Z}^n} e^{-2\pi \frac{i}{\hbar}  l\cdot \theta}e^{2\pi \frac{i}{\hbar}  l\cdot \omega},
   	\]
   	
   	Plugging in the last expression into the formula \eqref{neil} for the symbol representation of the operator $\Op_{\hbar}(\sigma_{\hbar})$ yields  
   	\begin{eqnarray*}
   		\text{Op}_{\hbar}(\sigma_{\hbar}) e_\omega(k)  &=& \int_{\mathbb{T}^n} e^{2\pi \frac{i}{\hbar}  k\cdot \theta} \sigma_{\hbar}(k,\theta) \widehat{ e_\omega}(\theta)\text{d}\theta\\
   		&=& \int_{\mathbb{T}^n} \sum_{l \in \hbar\mathbb{Z}^n}e^{2\pi \frac{i}{\hbar}  k\cdot \theta} \sigma_{\hbar}(k,\theta)\left[ e^{-2\pi \frac{i}{\hbar}  l\cdot \theta}e^{2\pi \frac{i}{\hbar}  l\cdot \omega}\right]\text{d}\theta\\
   		&=& \int_{\mathbb{T}^n} \sum_{l \in \hbar \mathbb{Z}^n}e^{-2\pi \frac{i}{\hbar}  (l-k)\cdot \theta} \sigma_{\hbar}(k,\theta) e^{2\pi \frac{i}{\hbar}  l\cdot \omega} \text{d}\theta\\
   		&=&  \sum_{l \in \hbar\mathbb{Z}^n}\widehat{ \sigma_{\hbar}}(k,l-k) e^{2\pi \frac{i}{\hbar}  l\cdot \omega} \\
   		&=&  \sum_{m \in \hbar\mathbb{Z}^n}\widehat{ \sigma_{\hbar}}(k,m)e^{2\pi \frac{i}{\hbar}  m\cdot \omega} e^{2\pi \frac{i}{\hbar}  k\cdot \omega}  \qquad (\text{where} \ m=l-k)\\
   		&=& \sigma_{\hbar}(k,\omega)e^{2\pi \frac{i}{\hbar}  k\cdot \omega},
   	\end{eqnarray*}
   	where $\widehat{\sigma_{\hbar}}$ stands for the toroidal Fourier transform of $\sigma_{\hbar}$ on the second variable, and for the last inequality we have used the formula \eqref{EQ:Finv}.  This gives the proof of  formula \eqref{B}. 
   \end{proof}


    \subsection{Semi-classical amplitudes}
    Writing out the semi-classical Fourier transform \eqref{gael1} as an inifite sum, suggests the following notation for the \textit{amplitude representation} of the pseudo-differential operator $ \text{Op}_{\hbar}(\sigma_{\hbar})$:
   \begin{equation}\label{3}
   \text{Op}_{\hbar}(\sigma_{\hbar})f(k) = : \sum_{m \in \hbar \mathbb{Z}^n}\int_{\mathbb{T}^n}e^{2\pi \frac{i}{\hbar}(k-m)\cdot \theta}\sigma_{\hbar}(k,\theta) f(m)\text{d}\theta\,.
   \end{equation}
    Let us point out that the right-hand side of \eqref{3} should not be regarded as an integral operator, but rather as an operator arising via formal integration by parts. This consideration allows performing operations like exchange of summation and integral. 

    Formula \eqref{3} gives rise to a possible generalisation where we allow the symbol $\sigma_{\hbar}$ to depend also on the variable $m \in \hbar \Z^n$; such functions $\sigma_{\hbar}$ shall be called \textit{semi-classical amplitudes}.   
   Formally we may also define operators of the form 
     \begin{equation}\label{AMP}
   Af(k) =\Op(a_{\hbar}) f(k)= \sum_{m \in \hbar \mathbb{Z}^n} \int_{\mathbb{T}^n} e^{2\pi \frac{i}{\hbar} (k-m)\cdot \theta} a_{\hbar}(k,m,\theta)f(m)\text{d}\theta\,,
   \end{equation}
where $a_{\hbar}: \hbar\Zn\times \hbar\Zn\times\Tn\to\C$, for all $f \in C^{\infty}(\hbar \mathbb{Z}^n)$.

\smallskip
In the next definition we extend  Definition \ref{DEF:symbolclasses} of the symbol classes  $S_{\rho,\delta}^\mu(\hbar\Z^n \times\T^n)$  to the semi-classical amplitudes depending on two lattice parameters say $k, m \in \hbar \mathbb{Z}^n$. The usefulness of this extended symbol classes, called \textit{amplitude classes}, becomes apparent in Theorem \ref{THM:adjoint} on the adjoint of a semi-classical pseudo-differential operator since its symbol is given in terms of an amplitude.  
     \begin{defn}[Amplitude classes $\mathcal{A}^{\mu_1, \mu_2}_{\rho, \delta} (\hbar\mathbb{Z}^n \times \hbar\Z^n \times \mathbb{T}^n)$]
     	Let $\rho,\delta\in\R$. The \textit{semi-classical amplitude class} $\mathcal{A}^{\mu_1,\mu_2}_{\rho,\delta} (\hbar\mathbb{Z}^n \times \hbar\mathbb{Z}^n \times \mathbb{T}^n) $ consists of the functions $a_{\hbar}:\hbar\Z^n\times\hbar\Z^n\times\T^n\to\C$ for which we have $a_{\hbar}(k,m, \cdot) \in C^\infty (  \mathbb{T}^n) $ for all $k,m \in \hbar\mathbb{Z}^n$,  provided that  for all multi-indices $\alpha,\beta,\gamma$ there exist a positive  $C_{\hbar,\alpha,\beta ,\gamma}>0$ such that for some $Q\in\mathbb{N}_0$ with $Q\leq |\gamma|$ we have
     	\begin{equation}\label{EQ:amps}
     		|D^{(\gamma)}_{\theta}\Delta^{\alpha}_{\hbar,k}\Delta^{\beta}_{\hbar,m} a_{\hbar}(k,m,\theta)| \leq C_{\alpha,\beta,\gamma}(1+|k|)^{\mu_1-\rho|\alpha| + \delta Q} (1+|m|)^{\mu_2-\rho|\beta| + \delta(|\gamma|-Q)}.
     	\end{equation} 

   \end{defn}

 Such a function $a_{\hbar}$ is called a \textit{semi-classical amplitude of order $(\mu_1,\mu_2)$ of type $(\rho,\delta)$}.  The operators with amplitudes in the  amplitude class $\mathcal{A}^{\mu_1,\mu_2}_{\rho,\delta} (\hbar\mathbb{Z}^n \times \hbar\mathbb{Z}^n \times \mathbb{T}^n)$ will be denoted by $\Op(\mathcal{A}^{\mu_1,\mu_2}_{\rho,\delta} (\hbar\mathbb{Z}^n \times \hbar\mathbb{Z}^n \times \mathbb{T}^n) )$. Moreover, by setting  $Q=|\gamma|$ in \eqref{EQ:amps} it is evident that   
 \[\Op(S_{\rho, \delta}^{\mu_1} ( \hbar\mathbb{Z}^n \times \mathbb{T}^n)) \subset \Op(\mathcal{A}^{\mu_1,\mu_2}_{\rho,\delta} (\hbar\mathbb{Z}^n \times \hbar\mathbb{Z}^n \times \mathbb{T}^n) )\,.\] 
    

\smallskip

 On the other hand semi-classical pseudo-differential operators arising from amplitudes are also pseudo-differential operators with symbols from some appropriate $S^{\mu}_{\rho,\delta}(\hbar \mathbb{Z}^n \times \mathbb{T}^n)$ class. In particular, we have the inclusion 
  $$\Op(\mathcal{A}^{\mu_1,\mu_2}_{\rho,\delta} (\hbar\mathbb{Z}^n \times \hbar \mathbb{Z}^n \times \mathbb{T}^n) )\subset \Op(S^{\mu_1+\mu_2}_{\rho, \delta}(\hbar\mathbb{Z}^n\times \mathbb{T}^n))\,,$$
  that is proven in Theorem \ref{THM:amplitudes}. For the proof of the latter, we first need an auxiliary result, see Lemma \ref{lemma1}, which in turn makes use of a relation between generalized difference operators, see Definition  \ref{DEF:gdiff} and the inverse of the Fourier transform  operator in our setting.  
 
       \begin{defn}[Generalised semi-classical  difference operators]\label{DEF:gdiff}Let $q\in C^\infty(\T^n)$. Then for $g:\hbar\Z^n\to\C$, the corresponding $q$-difference operator  is defined by 
       	\begin{equation}\label{deltaoperator12}
       		\Delta_{\hbar,q}g(k):= \frac{1}{\hbar} \int_{\mathbb{T}^n}e^{2\pi \frac{i}{\hbar} k\cdot \theta}q(\theta)\widehat{g}(\theta)\text{d}\theta.
       	\end{equation}
      
       \end{defn}
       Alternatively, one can get the following, useful for our purposes, expanded formula for \eqref{deltaoperator12} by writing out the Fourier transform of $g$ using \eqref{gael1}:
       \begin{equation}\label{deltaoperator2}
       	\Delta_{\hbar,q}g(k)=\frac{1}{\hbar} \sum_{l\in \hbar\mathbb{Z}^n}\int_{\mathbb{T}^n}e^{2\pi \frac{i}{\hbar} (k-l)\cdot \theta}q(\theta)g(l)\text{d}\theta=\frac{1}{\hbar} \sum_{l\in \hbar\mathbb{Z}^n} g(l) \mathcal{F}_{\hbar\Z^n}^{-1}q(k-l) = \frac{1}{\hbar}(g*\mathcal{F}_{\hbar \Zn}^{-1}q)(k).
       \end{equation}   
       Let us point out that as in the case of (standard) difference operators, see Definition \ref{DEF:diff.op}, the generalized $q$-difference operator can be extended to $g\in \mathcal{S}'(\hbar\Z^n)$. Finally, we note that the corresponding function $q$ does not have to be smooth, provided suitable behaviours of $g,q$. For instance formula \eqref{deltaoperator12} is well defined for $g\in\ell^2(\hbar\Z^n)$ and $q\in L^2(\T^n)$.
           
  
   We can now state the following result on the behaviour of the $\Delta_{\hbar,q}$ acting on symbols in the classes $S_{\rho,\delta}^{\mu}( \hbar \mathbb{Z}^n\times \mathbb{T}^n)$.
   \begin{lemma}\label{lemma1}
   Let $0\leq \delta\leq 1$ and $\mu\in\R$. Then, for $\sigma_{\hbar}\in S_{\rho,\delta}^{\mu}( \hbar \mathbb{Z}^n\times \mathbb{T}^n)$,  $q\in C^{\infty}(\mathbb{T}^n)$ and any $\beta\in \mathbb{N}_0^{n} $ we have  
   \begin{equation}\label{EQ:lemma1}
   |\Delta_{\hbar,q}D_{\hbar, \theta}^{(\beta)}\sigma_{\hbar}(k, \theta)|\leq  \frac{C_{q,\beta}}{\hbar} (1+|k|)^{\mu +\delta|\beta| }, 
   \end{equation}
  for all $k\in \hbar\mathbb{Z}^n$ and $\theta\in \mathbb{T}^n$. 
   \end{lemma}
   \begin{proof}
    Using the expression \eqref{deltaoperator2}, we can write
   \begin{eqnarray*}  \Delta_{\hbar,q}D_{\hbar,\theta}^{(\beta)}\sigma_{\hbar}(k, \theta) &=& \frac{1}{\hbar}\sum_{l \in \hbar \mathbb{Z}^n}\int_{\mathbb{T}^n}e^{2\pi \frac{i}{ \hbar}(k-l)\cdot \omega}q(\omega) D_{\hbar, \theta}^{(\beta)}\sigma_\hbar(l, \theta)\text{d}\omega \\
   &=&  \frac{1}{\hbar} D_{\hbar, \theta}^{(\beta)}\sigma_{ \hbar}(k, \theta)\int_{\mathbb{T}^n}q(\omega)\text{d}\omega  \\
   &+ &\frac{1}{\hbar}\sum_{\substack{
   l\in  \hbar \mathbb{Z}^n\\
   l\neq k}}\int_{\mathbb{T}^n}e^{2\pi \frac{i}{ \hbar}(k-l)\cdot \omega} q(\omega) D_{\hbar, \theta}^{(\beta)}\sigma_\hbar(l, \theta)\text{d}\omega \\
   &=:& T_1 + T_2,
   \end{eqnarray*} 
   
   where the first term is taken for $l=k$. Thus, for the first term we have
   $$|T_1|\leq  \frac{C_{q,\beta}}{\hbar} (1+|k|)^{\mu+\delta|\beta| }\,,$$
   by the assumption on the symbol $\sigma_{\hbar}$.
   Now, to estimate the second term, we first  assume that  $\beta \in \N_{0}^{n}$ is such that $\mu+\delta |\beta| \geq 0$. We rewrite the term $T_2$ in terms of the $M^\text{th}$ power, with $M$ to be chosen later, of the toroidial Laplace operator $\mathcal{L}_\omega$ using integration by parts and the formula \eqref{int.by.parts} as follows:
 \begin{equation}\label{EQ:seri}
\begin{aligned}
  |T_2| & = \frac{1}{\hbar} \left| \sum_{\substack{
   l\in  \hbar\mathbb{Z}^n\\
   l\neq k}} \int_{\T^n}\frac{e^{2\pi \frac{i}{ \hbar}(k-l)\cdot \omega}}{(2\pi \hbar^{-1})^{2M}|k-l|^{2M}}\left(\mathcal{L}_{\omega}^{M}q(\omega)  \right) D_{\theta}^{(\beta)}\sigma_{\hbar}(l,\theta) \text{d}\theta \right| \nonumber \\
&\leq  \frac{1}{\hbar}\hbar^{2M}C_{q,\beta}\sum_{\substack{
		l\in  \hbar\mathbb{Z}^n\\
		l\neq k}}\frac{1}{|k-l|^{2M}}(1+|l|)^{\mu+\delta |\beta|} \nonumber\\
   &\leq  \frac{1}{\hbar}\hbar^{2M} C_{q,\beta}\sum_{m\neq 0}\frac{1}{|m|^{2M}}(1+|k-m|)^{\mu+\delta |\beta|} \qquad (\text{where} \ l=k-m)\\
   &\leq  \frac{1}{\hbar}\hbar^{2M} C_{q,\beta}\sum_{m\neq 0}\frac{1}{|m|^{2M}}\Bigg((1+|k|)^{\mu+\delta |\beta|} + |m|^{\mu+\delta |\beta|} \Bigg) \nonumber \\
   & = \frac{1}{\hbar}C_{q,\beta}\sum_{\tilde{m}\neq 0}\frac{1}{|\tilde{m}|^{2M}}\Bigg((1+|k|)^{\mu+\delta |\beta|} + \hbar^{\mu+\delta|\beta|}|m|^{\mu+\delta |\beta|} \Bigg)  \quad (\text{where} \ m=\hbar \tilde{m})\\ 
   &\leq \frac{1}{\hbar}  C_{q,\beta} (1+|k|)^{\mu} \quad (\text{since}\ \hbar^{\mu+\delta|\beta|}\leq 1)\,,
\end{aligned}
\end{equation} 
where in the final estimate we have used the fact that  $\mu+ \delta |\beta| \geq 0$, and $M$ is chosen so that $\displaystyle M > \frac{n + \mu+\delta|\beta|}{2}$ which allows for the series above to converge.
  On the other hand, for the case where $\beta$ is such that $\mu+\delta |\beta|< 0$, we will use Peetre inequality; see  \cite[Proposition 3.3.31]{ruzhansky2009pseudo} which  under suitable considerations implies:
  
   $$
   (1+|k-m|)^{\mu+\delta |\beta|}\leq 2^{|\mu|+\delta|\beta|}(1+|k|)^{\mu+\delta |\beta|}(1+|m|)^{|\mu|+\delta |\beta|}\,,
   $$
   where $|\mu|$ stands for the absolute value of $|\mu|$.
   Hence, for $M$ such that $2M-|\mu|-\delta|\beta|>n$, and reasoning as above we have 
   \[
   |T_2| \leq  \frac{1}{\hbar}C_{q,\beta}\sum_{m\neq 0}\frac{1}{|m|^{2M}}(1+|k-m|)^{\mu+\delta |\beta|} \leq \frac{1}{\hbar} C_{q,\beta} (1+|k|)^{\mu+\delta |\beta|} \,.
   \]
   Summarising the above we have  
   \[
 |\Delta_{\hbar,q}D_{\hbar, \theta}^{(\beta)}\sigma_{\hbar}(k, \theta)| \leq \frac{C_{q,\beta}}{\hbar}(1+|k|)^{\mu+\delta |\beta|}+\frac{C_{q,\beta}}{\hbar} (1+|k|)^{\mu+\delta |\beta|} \leq \frac{C_{q,\beta}}{\hbar}(1+|k|)^{\mu+\delta |\beta|}\,,\]
   since $\hbar \leq 1$ and we have proved the estimate \eqref{EQ:lemma1} for all cases. The proof of Lemma \ref{lemma1} is now complete.
   \end{proof} 
 
       
Let us know present a toroidal Taylor expansion that is useful for our purposes, which is a re-scaled version of the toroidal Taylor expansion that  appeared in \cite[Theorem 3.4.4]{ruzhansky2009pseudo}:

\begin{thm}[Equivalent formulation of the toroidal Taylor expansion on $\T^n$]
For $f: \T^n \rightarrow \mathbb{C}$, $f \in C^{\infty}(\T^n)$,  we have the following equivalent formulation of the toroidal Taylor expansion:
\begin{equation}\label{EQ:Taylor}
	f(\theta)=\sum_{|\alpha|<N} \frac{\hbar^{-|\alpha|}}{\alpha !} (e^{2\pi \frac{i}{\hbar}\theta}-1)^\alpha D_{\hbar,\omega}^{(\alpha)} f(\omega)|_{\omega=0} +
	\sum_{|\alpha|=N} f_\alpha(\theta) (e^{2\pi \frac{i}{\hbar} \theta}-1)^\alpha\,,
\end{equation} 
where $D^{(\alpha)}_{\hbar, \omega}$ is given in \eqref{def.PD}. The functions $f_\alpha\in C^\infty(\T^n)$, where $|\alpha| \leq N$, are products of the  one-dimensional functions $f_j(\theta)$, $\theta\in\mathbb{T}$, defined inductively by

\begin{equation}\label{EQ:hjs}
f_{j + 1}(\theta):=
     \begin{cases}
       \frac{f_j(\theta) - f_j(0)}{e^{2\pi (i /\hbar) \theta} -1} &\quad\text{if}\ \ \theta \neq 0,\\
       D_{\hbar, \theta} f_j(\theta)\,,  &\quad\text{if}\ \theta = 0\,,
     \end{cases}
\end{equation}
 where we have set $f_0 :=f$.

\end{thm}
\begin{proof}For simplicity we will prove \eqref{EQ:Taylor} when $n=1$. In this case \eqref{EQ:Taylor} becomes 
\begin{equation}
    \label{Taylor,n=1}
    f(\theta)=\sum_{j=0}^{N-1} \frac{\hbar^{-j}}{j !} (e^{2\pi \frac{i}{\hbar}\theta}-1)^j D_{\hbar,\omega}^{(j)} f(\omega)|_{\omega=0} +
	 f_N(\theta) (e^{2\pi \frac{i}{\hbar} \theta}-1)^j\,.
\end{equation}
For $j \in \mathbb{N}_{0}$ we define
\begin{equation*}
f_{j + 1}(\theta):=
     \begin{cases}
       \frac{f_j(\theta) - f_j(0)}{e^{2\pi (i /\hbar) \theta} -1} &\quad\text{if}\ \ \theta \neq 0,\\
       D_{\hbar, \theta} f_j(\theta)\,,  &\quad\text{if}\ \theta = 0\,,
       \end{cases}
\end{equation*}
while if $j=0$, then we set $f_0(\theta):= f(\theta)$. Thus
\[
f_{j + 1}(\theta)=f_j(0)+f_{j+1}(\theta)(e^{2\pi (i /\hbar) \theta} -1)\,,
\]
while also we have
\begin{equation}
    \label{330}
    f(\theta)=\sum_{j=0}^{N-1}(e^{2\pi (i /\hbar) \theta} -1)^j f_j(0)+f_N(\theta) (e^{2\pi (i /\hbar) \theta} -1)^N\,.
\end{equation}
From the latter expression we see that it is enough to prove that 
\[
f_j(0)=\frac{\hbar^{-j}}{j!}D_{\hbar, \theta}^{(j)}f(\theta)|_{\theta=0}\,.
\]
It is clear that if $j<j_0$, then $D_{\hbar, \theta}^{(j)}(e^{2\pi (i /\hbar) \theta} -1)^{j_0}|_{\theta=0}=0 $. On the other hand when $j>j_0$, if we make the change of variable $\frac{\theta}{\hbar}=\tilde{\theta}$, then we have
\begin{eqnarray*}
D_{\hbar, \theta}^{(1)}(e^{2\pi (i /\hbar) \theta} -1)^{j_0} & = & \frac{1}{\hbar}    \left(\frac{1}{2 \pi i}\frac{\partial}{\partial \theta}-j_0 \right)\left(e^{2\pi \frac{i}{\hbar}\theta}-1 \right)^{j_0}\\
& = &  \left(\frac{1}{2 \pi i}\frac{\partial}{\partial \tilde{\theta}}-j_0 \right)\left(e^{2\pi i\tilde{\theta}}-1 \right)^{j_0} \\
& = &  j_0 \left( e^{2\pi i\tilde{\theta}}-1 \right)^{j_0-1}\,.
\end{eqnarray*}
The latter implies that 
\[
\left[ \prod_{i=1}^{j_0} \left( \frac{1}{2\pi i}\frac{\partial}{\partial \tilde{\theta}}-1\right)\right]\left(e^{2\pi i \tilde{\theta}}-1 \right)^{j_0}=j_0!\,,
\]
which in turn gives
\[
\left[ \prod_{i=0}^{j_0} \left( \frac{1}{2\pi i}\frac{\partial}{\partial \tilde{\theta}}-1\right)\right]\left(e^{2\pi i \tilde{\theta}}-1 \right)^{j_0}\big|_{\tilde{\theta}=0}=j_0!\,.
\]
Hence we get $\left[ \prod_{i=0}^{j} \left( \frac{1}{2\pi i}\frac{\partial}{\partial \tilde{\theta}}-1\right)\right]\left(e^{2\pi i \tilde{\theta}}-1 \right)^{j_0}\big|_{\tilde{\theta}=0}=j!\delta_{j,j_0}$, or after substitution, $D_{\hbar, \theta}^{(j)} \left(e^{2\pi \frac{i}{\hbar}\theta}-1 \right)^{j_0}\big|_{\theta=0}=\hbar^{-j}j! \delta_{j,j_0} $. 
   Finally an application of the operator $D_{\hbar, \theta}^{(j)}$ to both sides of the equality \eqref{330} gives $D_{\hbar, \theta}^{(j)}$, as desired, and the proof is complete. 
\end{proof}
 In the next result we see that the (semi-classical) amplitude representations of the form \eqref{AMP} are indeed (semi-classical) pseudo-differential operators. Particularly, if $a_\hbar \in \mathcal{A}^{\mu_1,\mu_2}_{\rho,\delta} (\hbar\mathbb{Z}^n \times \hbar\mathbb{Z}^n \times \mathbb{T}^n)$, then $\Op(a_\hbar)=\Op(\sigma_{\hbar,T})$ for some $\sigma_{\hbar,T} \in S^{\mu_1+\mu_2}_{\rho,\delta}(\hbar \Z^n \times \T^n)$. Formally we have:  

\begin{thm}\label{THM:amplitudes}
 	Let $0\leq \delta<\rho \leq 1$. For $a_{\hbar}\in \mathcal{A}^{\mu_1,\mu_2}_{\rho,\delta} (\hbar\mathbb{Z}^n \times \hbar\mathbb{Z}^n \times \mathbb{T}^n)$ let the corresponding amplitude operator $T$ be given by
 \begin{equation}\label{op.T}
 	Tf(k) = \sum_{l \in \mathbb{Z}^n} \int_{\mathbb{T}^n} e^{2\pi \frac{i}{\hbar} (k-l)\cdot \theta} a_{\hbar}(k,l,\theta)f(l)\text{\rm d}\theta.
 \end{equation}
 Then we have $T = \Op(\sigma_{\hbar,T })$ for some $\sigma_{\hbar,T } \in S^{\mu_1+\mu_2}_{\rho,\delta} (\hbar \mathbb{Z}^n \times  \mathbb{T}^n) $. Moreover, 
 \begin{equation}\label{EQ:expamp}
 	\sigma_{\hbar,T } (k,\theta) \sim \sum_{\alpha } \frac{1}{\alpha !}\Delta_{\hbar,l}^{\alpha} D_{\hbar, \theta}^{(\alpha)} a_{\hbar } (k,l,\theta)\Big|_{l=k}\,;
 \end{equation} 
  that is  for all $N\in\mathbb{N}$ we have 
 \begin{equation}\label{EQ:expamp2}
 	\sigma_{\hbar,T } - \sum_{|\alpha| < N } \frac{1}{\alpha !}\Delta_{\hbar,l}^{\alpha} D_{\hbar,\theta}^{(\alpha)} a_{\hbar } (k,l,\theta)\Big|_{l=k}  \in  S^{\mu_1+\mu_2 -N(\rho-\delta)}_{\rho,\delta} ( \hbar\mathbb{Z}^n \times  \mathbb{T}^n)\,.
 \end{equation} 
\end{thm}

\begin{proof}[Proof of Theorem \ref{THM:amplitudes}]
We will apply Proposition
 \ref{PROP:symbols}, to find the formula for the symbol $\sigma_{\hbar,T}$ of the operator $T$ as in the hypothesis.  We have
\begin{eqnarray}\label{b.Tayl}
	\sigma_{\hbar,T} (k,\theta) &=& e^{-2\pi \frac{i}{\hbar} k\cdot \theta}\sum_{l \in \hbar\mathbb{Z}^n}\int_{\mathbb{T}^n} e^{2\pi \frac{i}{\hbar} (k-l)\cdot \omega} a_\hbar(k,l,\omega) e^{2\pi \frac{i}{\hbar} l\cdot \theta} \text{d}\omega\nonumber\\
	&=& \sum_{l \in \hbar \mathbb{Z}^n}\int_{\mathbb{T}^n} e^{2\pi \frac{i}{\hbar} (k-l)\cdot (\omega-\theta)} a_{\hbar}(k,l,\omega)\text{d}\omega\nonumber\\
	&=& \int_{\mathbb{T}^n} e^{2\pi \frac{i}{\hbar} k\cdot (\omega-\theta)} \widehat{ a_{\hbar}}(k,\omega-\theta,\omega)\text{d}\omega\nonumber\\
 & = &  \int_{\mathbb{T}^n} e^{2\pi \frac{i}{\hbar} k\cdot \omega} \widehat{ a_{\hbar}}(k,\omega,\omega+\theta)\text{d}\omega\,, 
\end{eqnarray}
where $\widehat{a_\hbar}$ stands for the the semi-classical Fourier transform of $a_{\hbar}$ with respect to the second variable, and in the last inequality we have replaced $\omega-\theta$ by $\omega$. Now if we take the Taylor expansion \eqref{EQ:Taylor} in our setting of $\widehat{a}_{\hbar}(k, \omega, \omega+\theta)$ in the third variable $\theta$ we get 
\begin{equation}
    \label{afterT.exp}
    \widehat{a}_{\hbar}(k,\omega, \omega+\theta)=\sum_{|\alpha| \leq N}\frac{\hbar^{-|\alpha|}}{\alpha !} \left(e^{2\pi \frac{i}{\hbar}\omega}-1 \right)^{\alpha}D_{\hbar, \theta}^{(\alpha}\widehat{a}_{\hbar}(k, \omega, \theta)+R_0\,,
\end{equation}
where $R_0$ is the remainder term and will be analysed in the sequel. Plugging \eqref{afterT.exp} into \eqref{b.Tayl} we obtain 
\begin{equation}
	\label{E}
	\sigma_{\hbar,T} (k,\theta)  =\int_{\mathbb{T}^n} e^{2\pi \frac{i}{\hbar} k\cdot \omega} \sum_{|\alpha| \leq N} \frac{\hbar^{-|\alpha|}}{\alpha !}(e^{2\pi \frac{i}{\hbar}\omega} -1 )^\alpha D_{\hbar, \theta}^{(\alpha)} \widehat{a_{\hbar}}(k,\omega,\theta)\text{d}\omega + R\,,
\end{equation}
with $R$ in terms of $R_0$. Now since from \eqref{EQ:diffs} we have

\[\Delta^{\alpha}_{\hbar} g (k)=\frac{1}{\hbar^{|\alpha|}}\int_{\mathbb{T}^n} e^{2\pi \frac{i}{\hbar} k\cdot y} (e^{2\pi \frac{i}{\hbar}  y} -1 )^\alpha  \widehat{g}(y)\text{d}y, \]
we can obtain the following alternative expression for $\sigma_{\hbar, T}$ as follows by \eqref{E}
\[ \sigma_{\hbar,T} (k,z) = \sum_{|\alpha|\leq  N} \frac{1}{\alpha !}\Delta^{\alpha}_{\hbar,l} D_{\hbar, \theta}^{(\alpha)} a_{\hbar} (k,l,\theta)\Big|_{l=k} + R\,,\]
which shows \eqref{EQ:expamp}. Now, proving \eqref{EQ:expamp2} amounts to analysing the remainder $R$, which is a sum of terms of the form 
\[
R_j(k, \theta)= \int_{\mathbb{T}^n} e^{2\pi \frac{i}{\hbar} k\cdot y} (e^{2\pi \frac{i}{\hbar}  y} -1 )^\alpha  b_{\hbar,j}(k,y,z)\text{d}y, \]
where $|\alpha| = N$ and the $b_j$'s are with the use of \eqref{EQ:hjs}  combinations of functions 
\[
D_{\hbar, \theta}^{\alpha_0}\mathcal{F}_{2}a_{\hbar}(k,y, z)
\]
for some $|\alpha_0| \leq N$, where $\mathcal{F}_2$ stands for the Fourier transform with respect to the second variable multiplied by some smooth functions $a_j$. Consequently, for any $\beta$ we have that  $D_{z}^{(\beta)} R_j(k,z)$ are the sums of terms with  the form 
\[
\int_{\mathbb{T}^n} e^{2\pi \frac{i}{\hbar} k\cdot y} a_j(y) (e^{2\pi \frac{i}{\hbar}  y} -1 )^\alpha  D_{\hbar, z}^{(\beta)}D_{\hbar,z}^{(\alpha_0)}\mathcal{F}_2 a_\hbar(k,y,z) \text{d}y\,,
\]
which in turn implies that $D_{\hbar, z}^{(\beta)} R_j(k,z)$ are the sums of terms of the form 
\[
\hbar \Delta_{\hbar, a_j}\Delta_{\hbar, l}^{\alpha}D_{\hbar, z}^{(\alpha_0+\beta)}a_{\hbar}(k,l,z)\Big|_{l=k}\,,\]
where the factor $\hbar$ is due to the definition \eqref{deltaoperator12}. 
Now, since 
$\displaystyle a_\hbar \in\mathcal{A}^{\mu_1,\mu_2}_{\rho,\delta} $, an application of 
  Lemma \ref{lemma1} yields that   that $\displaystyle  R_j(k,z)$ satisfies
\[|R_j(k,z)| \leq  C(1 +\big|k\big|)^{\mu_1  } (1 +\big|k\big|)^{\mu_2 -\rho|\alpha| + \delta|\alpha_0| + \delta|\beta|}\,,\]
for some $\beta$ as in \eqref{EQ:amps}.
Taking  $|\alpha| = N$ and $|\alpha_0|\leq N$
we obtain that  
\[|R_j(k,z)| \leq C(1 +\big|k\big|)^{\mu_1+\mu_2 -(\rho -\delta)N + \delta|\beta|  }.\]
On the other hand, the terms $\Delta^{\beta}_{\hbar,k} R_j(k,z)$  can be represented as  sums of terms of the form
$$
\int_{\mathbb{T}^n} e^{2\pi \frac{i}{\hbar} k\cdot w}  (e^{2\pi i  w} -1 )^\beta a_j(w)(e^{2\pi \frac{i}{\hbar} w} - 1)^{\alpha} b_j(k,w,z) \text{d}w,
$$
where $b_j$ and $a_j$ as of the above form. Similar to the above arguments show that 
\[|\Delta^{\beta}_{\hbar,k} R_j(k,z)| \leq C(1 +\big|k\big|)^{\mu_1 +\mu_2 -\rho |\beta|  -(\rho-\delta)N }.\]
Taking  $N$ large enough, and following standard arguments  from the classical pseudo-differential calculus we deduce the expansion  \eqref{EQ:expamp} and the proof is complete.
\end{proof}

   \section{Semi-classical symbolic caclulus}
\label{SEC:calculus}
   
         In this section we establish the symbolic calculus of semi-classical pseudo-differential operators on $\hbar \Z^n$. In particular we develop the  formulae for the composition of operators, adjoint and transpose operators. In the end of the section we establish the notion of ellipticity in our setting. 
         
     \begin{thm}[Composition formula for $\Psi_{\hbar}DO$s]\label{THM:comp}
    Let $0\leq \delta<\rho\leq 1.$ Let $\sigma_{\hbar} \in S^{\mu_1}_{\rho, \delta}( \hbar \mathbb{Z}^n\times\mathbb T^n)$ and $\tau_{\hbar}\in S^{\mu_2}_{\rho, \delta}(\hbar \mathbb{Z}^n\times\mathbb T^n)$. Then the composition $\rm{Op}(\sigma_{\hbar})\circ \rm{Op}(\tau_{\hbar})$ is a pseudo-differential operator with symbol $\varsigma_{\hbar}\in S^{\mu_1 +\mu_2}_{\rho, \delta}(\hbar\mathbb{Z}^n\times\mathbb T^n)$, given by
  the asymptotic sum     \begin{equation}\label{EQ:comp}
       	\varsigma_{\hbar}(k, \theta)\sim  \sum_{\alpha} \frac{1}{\alpha !} D_{\hbar, \theta}^{(\alpha)} \sigma_{\hbar}(k, \theta)\Delta_{\hbar,k}^{\alpha}\tau_{\hbar}(k,\theta).
       \end{equation}
   \end{thm}
  Observe that the order of taking differences and derivatives in \eqref{EQ:comp} is different from  the analogous composition formulae on the classical cases $\Rn$ and $\Tn$, see \cite{RT-JFAA,ruzhansky2009pseudo}.

\begin{proof}[Proof of Theorem \ref{THM:comp}] 
The semi-classical pseudo-differential operators with symbols $\sigma_{\hbar}$ and $\tau_{\hbar}$ are given respectively by
\begin{equation*}
	\label{9}
	\text{Op}_{\hbar}(\sigma_{\hbar})f(k) = \sum_{m \in \hbar\mathbb{Z}^n}\int_{\mathbb{T}^n}e^{2\pi \frac{i}{\hbar}(k-m)\cdot \theta}\sigma_{\hbar}(k,\theta)f(m)\text{d}\theta,
\end{equation*}
\begin{equation*}
	\label{10}
	\text{Op}_{\hbar}(\tau_{\hbar})g(m) = \sum_{l \in \hbar\mathbb{Z}^n}\int_{\mathbb{T}^n}e^{2\pi \frac{i}{\hbar}(m-l)\cdot \omega}\tau_{\hbar}(m,\omega) g(l)\text{d}\omega\,,
\end{equation*}
where $f ,g \in \mathcal{S}(\hbar\mathbb{Z}^n)$.
Consequently we have

\begin{eqnarray*}
	\text{Op}_{\hbar}(\sigma_{\hbar})\big(\text{Op}_{\hbar}(\tau_{\hbar})g\big)(k)  	&=& \sum_{m \in \hbar\mathbb{Z}^n}\int_{\mathbb{T}^n}e^{2\pi \frac{i}{\hbar}(k-m)\cdot \theta}\sigma_{\hbar}(k,\theta)	\text{Op}(\tau_{\hbar})g(m) \text{d}\theta
 \\
	&=& \sum_{m \in \hbar\mathbb{Z}^n}\int_{\mathbb{T}^n}e^{2\pi \frac{i}{\hbar}(k-m)\cdot \theta}\sigma_{\hbar}(k\theta)	\left[\sum_{l \in \hbar\mathbb{Z}^n}\int_{\mathbb{T}^n}e^{2\pi \frac{i}{\hbar}(m-l)\cdot \omega}\tau_{\hbar}(m,\omega) g(l)\text{d}\omega\right] \text{d}\theta \\
	&=&  \sum_{l \in \hbar\mathbb{Z}^n}\sum_{m \in \hbar \mathbb{Z}^n}\int_{\mathbb{T}^n}\int_{\mathbb{T}^n}e^{2\pi \frac{i}{\hbar}(k-m)\cdot \theta}\sigma_{\hbar}(k,\theta) e^{2\pi \frac{i}{\hbar}(m-l)\cdot \omega}\tau_{\hbar}(m,\omega) g(l)\text{d}\omega\text{d}\theta\\
	&=& \sum_{l \in \hbar\mathbb{Z}^n} \int_{\mathbb{T}^n}e^{2\pi \frac{i}{\hbar}(k-l)\cdot \omega}\varsigma_{\hbar}(k,\omega)g(l) \text{d}\omega,
\end{eqnarray*}
where 
\begin{eqnarray*}
	\varsigma_{\hbar}(k,\omega) & \:=&\sum_{m \in \hbar\mathbb{Z}^n}\int_{\mathbb{T}^n} e^{2\pi \frac{i}{\hbar}(k-m)\cdot(\theta- \omega)}\sigma_{\hbar}(k,\theta) \tau_{\hbar}(m,\omega)\text{d}\theta \\
	&=&\sum_{m \in \hbar\mathbb{Z}^n}\int_{\mathbb{T}^n}e^{2\pi \frac{i}{\hbar}k\cdot(\theta- \omega)}e^{-2\pi \frac{i}{\hbar}m\cdot(\theta- \omega
 )}\sigma_{\hbar}(k,\theta) \tau_{\hbar}(m,\omega) \text{d}\theta \\
	&=& \int_{\mathbb{T}^n}e^{2\pi \frac{i}{\hbar}k\cdot(\theta- \omega)}\sigma_{\hbar}(k,\theta) \widehat{\tau_{\hbar}}(\theta- \omega,\omega)\text{d}\theta  \\
	&=& \int_{\mathbb{T}^n} e^{2\pi \frac{i}{\hbar}k\cdot \theta}\sigma_{\hbar}(k,\omega + \theta) \widehat{\tau_{\hbar}}(\theta,\omega) \text{d}\theta  \qquad (\text{replace} \ \theta-\omega \ \text{by} \ \theta)\,,
\end{eqnarray*}
where $\widehat{\tau_{\hbar}}$ denotes the (semi-classical) Fourier transform of $\tau_{\hbar}(m, \omega)$ in the first variable.

Employing the toroidal Taylor expansion given by \eqref{EQ:Taylor} on the symbol $\sigma_{\hbar}(k,\omega+\theta)$ gives 
  
\begin{eqnarray*}
	\varsigma_{\hbar}(k,y) &=& \int_{\mathbb{T}^n} e^{2\pi \frac{i}{\hbar}k\cdot \omega}\sigma_{\hbar}(k,\theta + \omega) \widehat{\tau_{\hbar}}(\omega,\theta) \text{d}\omega\\
	&=& \int_{\mathbb{T}^n} e^{2\pi \frac{i}{\hbar}k\cdot \omega}  \sum_{|\alpha|< N} \frac{\hbar^{-|\alpha|}}{\alpha !} (e^{2\pi \frac{i}{\hbar} \omega}- 1)^{\alpha} D_{\hbar, \theta}^{(\alpha)} \sigma_{\hbar}(k, \theta) \widehat{\tau_{\hbar}}(\omega,\theta) \text{d}\omega  + R,\\ 
	&=& \sum_{|\alpha|< N} \frac{1}{\alpha !} D_{\hbar,\theta}^{(\alpha)} \sigma_{\hbar}(k, \theta)\Delta_{\hbar,k}^{\alpha}\tau_{\hbar}(k,\theta) + R,
\end{eqnarray*} 
where $R$ is a remainder  from the Taylor expansion and for the last equality we have used the expression \eqref{EQ:diffs}. Now, since the difference operators satisfy the Leibniz rule we get
\[
|D_{\hbar,\theta}^{(\alpha)} \sigma_{\hbar}(k, \theta)\Delta_{\hbar,k}^{\alpha}\tau_{\hbar}(k,\theta)|\leq C_{\alpha} (1+|k|)^{\mu_2+\delta|\alpha|}(1+|k|)^{\mu_1-\rho |\alpha|}\,,
\]
that is we have 
$$D_{\hbar, \theta}^{(\alpha)} \sigma_{\hbar}(k, \theta)\Delta_{\hbar,k}^{\alpha}\tau_{\hbar}(k,\theta)\in S^{\mu_1+\mu_2-(\rho-\delta)|\alpha|}_{\rho,\delta}(\hbar\Z^n\times\T^n).$$
Finally, to estimate the remainder $R$  we follow lines of Theorem \ref{THM:amplitudes}.
\end{proof}

In the following theorem we prove that in the lattice case $\hbar \Z^n$ we still have the desired property for the adjoint of a semi-classical pseudo-differential operator. Before doing so, let us point out that the adjoint operator makes sense in our setting  for operators acting on the Hilbert space $\ell^2(\hbar \Z^n)$.

  \begin{thm}[Adjoint of a $\Psi_{\hbar}DO$]\label{THM:adjoint}
  Let $0\leq \delta<\rho\leq 1.$ Let $\sigma_{\hbar}\in S^{\mu}_{\rho, \delta}(\hbar\mathbb{Z}^n\times\mathbb T^n)$. Then there exist a symbol $\sigma_{\hbar}^*\in S^{\mu}_{\rho, \delta}(\hbar\mathbb{Z}^n\times\mathbb T^n)$ such that the adjoint operator $\rm{Op}(\sigma_{\hbar})^*$ is a pseudo-difference operator with symbol $\sigma_{\hbar}^{*}$;  that is we have  $\rm{Op}(\sigma_{\hbar})^* = \rm{Op}(\sigma_{\hbar}^*)$. Moreover, we have the asymptotic expansion 
  \begin{equation}\label{EQ:adj}
  	\sigma_{\hbar}^*(k,\theta)  \sim \sum_\alpha \frac{1}{\alpha !}\Delta_{\hbar,k}^{\alpha} D_{\hbar, \theta}^{(\alpha)}\overline{\sigma_{\hbar}(k,\theta)}. 
  \end{equation}   
\end{thm}
\begin{proof}

Let $f,g \in \ell^2(\hbar \Z^n)$. We have 
\begin{eqnarray*}
	\big( \text{Op}_{\hbar}(\sigma_{\hbar})f,g \big)_{\ell ^2(\hbar\mathbb{Z}^n)} &=&  \sum_{k \in \hbar\mathbb{Z}^n} \text{Op}(\sigma_{\hbar})f(k) \overline{g(k)}\\
	&=& \sum_{k \in \hbar\mathbb{Z}^n}\sum_{l \in \hbar\mathbb{Z}^n}\int_{ \mathbb{T}^n} e^{2\pi \frac{i}{\hbar} (k-l)\cdot \theta}\sigma_{\hbar}(k,\theta)  f(l) \overline{g(k)} \text{d}\theta \\
	&=& \sum_{l \in \hbar\mathbb{Z}^n} f(l)\Bigg(\overline{ \sum_{k \in \hbar\mathbb{Z}^n} \int_{ \mathbb{T}^n}e^{-2\pi \frac{i}{\hbar} (k-l)\cdot \theta} \overline{\sigma_{\hbar}(k,\theta)} g(k)} \text{d}\theta\Bigg)\\
	&=& \sum_{l \in \hbar\mathbb{Z}^n} f(l)\overline{\text{Op}(\sigma_{\hbar})^* g(l) }\,.
\end{eqnarray*}
By the definition of the adjoint operator we must have 
\[ \text{Op}_{\hbar}(\sigma_{\hbar})^* g(l) = \sum_{k \in \hbar\mathbb{Z}^n} \int_{ \mathbb{T}^n}e^{-2\pi \frac{i}{\hbar} (k-l)\cdot \theta} \overline{\sigma_{\hbar}(k,\theta)} g(k) \text{d}\theta\,,\]
or interchanging $k$ with $l$
\[ \text{Op}_{\hbar}(\sigma_{\hbar})^* g(k) = \sum_{l \in \hbar\mathbb{Z}^n} \int_{ \mathbb{T}^n}e^{2\pi  \frac{i}{\hbar} (k-l)\cdot \theta} \overline{\sigma_{\hbar}(l,\theta)} g(l) \text{d}\theta\,,\]
i.e. $\text{Op}_{\hbar}(\sigma_{\hbar})^*$  is an amplitude operator with amplitude \[a_{\hbar}(k,l,\theta) = \overline{\sigma_{\hbar}(l,\theta)}\in S^{\mu}_{\rho,\delta}(\hbar \Z^n,\T^n)=\mathcal{A}^{0,\mu}_{\rho,\delta}(\hbar\Z^n\times \hbar\Z^n\times\T^n)\,.\] Hence $\text{Op}_{\hbar}(\sigma_{\hbar})^* = \text {Op}_{\hbar}(\sigma_{\hbar}^*)$, and
  Theorem \ref{THM:amplitudes} yields the asymptotic expansion
\[\sigma_{\hbar}^*(k,\theta)  \sim \sum_{\alpha} \frac{1}{\alpha !}\Delta_{\hbar,l}^{\alpha} D_{\hbar,\theta}^{(\alpha)}\overline{\sigma_{\hbar}(l,\theta)}\Big|_{l=k}\,.\]
The proof of Theorem \ref{THM:adjoint} is now complete.
\end{proof}
Before turning over to analyse the transpose (or algebraic adjoint) operator in our setting let us first recall how the transpose operator reads in our setting:

\medskip 
  	For $f, g\in \mathcal{S}(\hbar\mathbb{Z}^n)$, the transpose $T^t$ of a linear operator $T$ satisfies the distributional duality
  	\[\langle T^t f,g\rangle =\langle f,T g\rangle\,;\] 
  	   that is for $k\in \hbar\mathbb{Z}^n$ we have the equality
  	\[\sum_{k\in \hbar\mathbb{Z}^n}(T^t f)(k)g(k) = \sum_{k\in \hbar\mathbb{Z}^n}f(k)(Tg )(k). \]

  \begin{thm}[Transpose of a $\Psi_{\hbar}DO$]\label{THM:transpose}
  	Let $0\leq \delta<\rho\leq 1$ and let $\sigma_{\hbar}\in S^{\mu}_{\rho, \delta}(\hbar\mathbb{Z}^n\times\mathbb T^n)$. Then there exists a symbol $\sigma_{\hbar}^t\in S^{\mu}_{\rho, \delta}(\hbar\mathbb{Z}^n\times\mathbb T^n)$ so that the transpose operator $\rm{Op}(\sigma_{\hbar})^t$ is a semi-classical pseudo-differential operator with symbol $\sigma_{\hbar}^t$; i.e., we have $\rm{Op}(\sigma_{\hbar})^t = \rm{Op}(\sigma_{\hbar}^t)$. The asymptotic formula for the symbol $\sigma_{\hbar}^{t}$ is given by
  	\begin{equation}\label{EQ:transpose}
  		{\sigma_{\hbar}^t}(k, \theta) \sim \sum_{\alpha}\frac{1}{\alpha!}\Delta_{\hbar,k}^{\alpha}D_{\hbar, \theta}^{(\alpha)}\sigma_{\hbar}(k, -\theta)\,.
  	\end{equation}
  \end{thm}
  
  \begin{proof}
  We have
  	\begin{eqnarray*}
  		\sum_{k \in \hbar\mathbb{Z}^n}f(k)(Tg )(k)&= & \sum_{k \in \hbar\mathbb{Z}^n} \sum_{l\in \hbar\mathbb{Z}^n}\int_{\mathbb{T}^n} f(k)e^{2\pi \frac{i}{\hbar}(k-l)\cdot \theta}\sigma_{\hbar}(k, \theta)g(l)\text{d}\theta\\
  		&= & \sum_{l \in \hbar\mathbb{Z}^n}g(l)\Bigg(\sum_{k\in \hbar\mathbb{Z}^n}\int_{\mathbb{T}^n}e^{2\pi \frac{i}{\hbar}(k-l)\cdot \theta}\sigma_{\hbar}(k, \theta)f(k)\text{d}\theta\Bigg).
  	\end{eqnarray*}
  	By the definiton of the transpose operator we must have
  	\[ T^t g(l) =\sum_{k\in \hbar\mathbb{Z}^n}\int_{\mathbb{T}^n}e^{ 2\pi \frac{i}{\hbar}(k-l)\cdot \theta}\sigma_{\hbar}(k, \theta)g(k)\text{d}\theta\,, \]
  	or equivalently
  	\[ T^t g(l) =\sum_{k\in \hbar\mathbb{Z}^n}\int_{\mathbb{T}^n}e^{2\pi \frac{i}{\hbar}(l-k)\cdot \theta}\sigma_{\hbar}(k, -\theta)g(k)\text{d}\theta\,. \] 
   The last formula corresponds to an amplitude operator with 
  	amplitude $a_{\hbar}(l, k, \theta) =\sigma_{\hbar}(k, -\theta) \in \mathcal{A}^{0,\mu}_{\rho,\delta} (\hbar\mathbb{Z}^n \times \times \hbar\Z^n \times \mathbb{T}^n) $.
  	Hence $\displaystyle T^t = \text{Op}(a_{\hbar}^t)$, and
  	 Theorem \ref{THM:amplitudes} gives 
  	\begin{equation*}
  		{\sigma_{\hbar}^t}(k, \theta) \sim \sum_{\alpha}\frac{1}{\alpha!}\Delta_{\hbar,k}^{\alpha}D_{\hbar,\theta}^{(\alpha)}\sigma_{\hbar}(k, -\theta)\,.
  	\end{equation*}
  	The proof of Theorem \ref{THM:transpose} is now complete.
  \end{proof}

The following result on the asymptotic sums of symbols is well-known in the classical cases as well. This is a utility tool that can simplify the process of solving the so-called ``eliptic'' partial differential equations.
   \begin{lemma}[Asymptotic sums of symbols of $\Psi_{\hbar}DO$s]\label{LEM:as.sums}
   Let $1\geq \rho>\delta \geq 0$, and let $\big\{\mu_j\big\}_{j=0}^{\infty} \subset \mathbb{R}$ be a decreasing sequence such that $\mu_j\rightarrow -\infty$ as $j\rightarrow \infty$.
   If $\sigma_{\hbar,j} \in S_{\rho, \delta}^{\mu_j}(\hbar \mathbb{Z}^n\times\mathbb{T}^n)$  for all $j\in \mathbb{N}_0$, then there exists  $\sigma_\hbar \in S_{\rho, \delta}^{\mu_0}(\hbar \mathbb{Z}^n\times\mathbb{T}^n)$ such that 
   \[\sigma_{\hbar} \sim \sum_{j=0}^{\infty}\sigma_{\hbar,j},\]
   that is for all $N \in \mathbb{N}$ we have 
   \[\sigma_{\hbar} - \sum_{j=0}^{N-1}\sigma_{\hbar,j} \in S_{\rho, \delta}^{\mu_N}(\hbar \mathbb{Z}^n\times\mathbb{T}^n)\,. \]
   \end{lemma}
   \begin{proof}
    The proof is a direct consequence of   \cite[Theorem 4.4.1]{ruzhansky2009pseudo} taking into account that the symbol classes there are the same modulo swapping the order of variables.  
   \end{proof}

Let now demonstrate  the notion of ellipticity of  semi-classical pseudo-differential  operators with symbol in the symbol classes $S^{\mu}_{\rho, \delta}(\hbar \mathbb{Z}^n\times\mathbb T^n)$. The following definition is an adaptation of the notion of ellipticity in the classical settings.
   
        \begin{defn}[Elliptic operators]
      A symbol  $\sigma_\hbar \in S^{\mu}_{\rho, \delta}(\hbar \mathbb{Z}^n\times\mathbb T^n)$ shall be called  \textit{elliptic} (of order $\mu$) if there exist $C >0$ and $M>0$ such that 
      \[|\sigma_\hbar
      (k,\theta)|\geq C (1+|k|)^{\mu}\]
      for all  $\theta \in \mathbb{T}^n$ and for  $|k|\geq M$, $k\in \hbar\mathbb{Z}^n$. The semi-classical pseudo-differential operators with elliptic symbols  shall  also be called elliptic.
      \end{defn}

   In the next result we show that the ellipticity of the pseudo-differential operator on $\hbar \Z^n$ is equivalent, as it happens also in the classical cases, to the invertibility of it in the algebra of operators $\Op(S^{\infty}(\hbar \Z^n \times \T^n))/ \Op(S^{-\infty}(\hbar \Z^n \times \T^n))$.\footnote{As usually, we define $S^{\infty}(\hbar \Z^n \times \T^n)=\bigcup_{\mu \in \R}S^{\mu}_{1,0}(\hbar \Z^n \times \T^n)$.}
   
   \smallskip The notion of a  \textit{parametrix} in our setting follows the lines of the general theory and reads as:
   
   \begin{defn}[parametrix]
   The operator $T$ is called the right (resp. left) \textit{parametrix} of $S$ if $ST-I \in \Op(S^{-\infty}(\hbar \mathbb{Z}^n\times\mathbb{T}^n))$ (resp. $TS-I \in \Op(S^{-\infty}(\hbar \mathbb{Z}^n\times\mathbb{T}^n)) $), where $I$ is the identity operator \footnote{The notion of a parametrix is applicable to all pseudo-differential operators on $\hbar \Z^n$}. 
   \end{defn}
   \begin{thm}[The ellipticity of a $\Psi_\hbar$DO is equivalent to the existence of its parametrix]\label{THM:elliptic}
   Let $0\leq \delta<\rho \leq 1.$
   An operator $U\in \Op(S^{\mu}_{\rho, \delta}(\hbar \Z^n\times\mathbb{T}^n))$ is elliptic if and only if there exists $V\in \Op(S^{-\mu}_{\rho, \delta}(\hbar \Z^n\times\mathbb{T}^n))$ such that 
   \[VU\backsim I\backsim UV\;\textrm{ modulo }\; \Op(S^{-\infty}(\hbar \Z^n\times \mathbb T^n))\,,\]
   i.e., the operator $V$ is the left and right parametrix of $U$.
   \smallskip
   
   Moreover, let $\displaystyle U\sim \sum_{l=0}^{\infty}U_l$ be an expansion of the operator $U$, where \[U_l\in \Op(S_{\rho,\delta}^{\mu-(\rho-\delta)l}(\hbar \Z^n\times \mathbb{T}^n))\,.\] Then the corresponding asymptotic expansion  of the operator $V$ can be expressed via $\displaystyle V\sim \sum_{j=0}^{\infty}V_j$ with \[V_j\in \Op(S_{\rho,\delta}^{-\mu-(\rho-\delta)j}(\hbar \Z^n\times \mathbb{T}^n))\]
    can be obtained by setting $\displaystyle \sigma_{\hbar,V_0} := \frac{1}{\sigma_{\hbar,U_0}}$, and then recursively    
   \begin{equation}\label{EQ:parexp}
 \sigma_{\hbar,V_N}(k, \theta) = \frac{-1}{\sigma_{\hbar,U_0}(k, \theta)} \sum_{j = 0}^{N-1}\sum_{l = 0}^{N-1}\sum_{|\gamma| = N - j - l}\frac{1}{\gamma!}\Big[D_{\hbar, \theta}^{(\gamma)} \sigma_{\hbar,V_j}(k, \theta)\Big] \Delta_{\hbar,k}^{\gamma}\sigma_{\hbar,U_l}(k, \theta).
\end{equation}
   \end{thm}

   \begin{proof}
   First we want to prove that given that the existence of an operator $V$ as in the statement satisfying 
   \[I - UV = T\in S^{-\infty}(\hbar \Z^n\times\mathbb{T}^n)\,,\]  the ellipticity of the operator $U\in $\ \rm{Op}$(S^{\mu}_{\rho, \delta}(\hbar \Z^n\times\mathbb{T}^n))$ can be deduced. 
   By the composition formula, see  Theorem \ref{THM:comp}, we get 
   \[ 1 - \sigma_{\hbar,U}(k, \theta)\sigma_{\hbar,V}(k, \theta)\in S^{-(\rho-\delta)}_{\rho, \delta}(\hbar \Z^n\times\mathbb{T}^n)\,.\]
   The latter means that there exists a constant $C >0$ such that
   \[ |1 - \sigma_{\hbar,U}(k, \theta)\sigma_{\hbar,V}(k, \theta)| \leq  C (1+|k|)^{-(\rho-\delta)}.\]
  By choosing $M$ so that $\displaystyle C_\hbar (1+|M|)^{-(\rho - \delta)}< \frac{1}{2}$, the last estimate yields
   \begin{equation}
   \label{M}
   |\sigma_{\hbar,U}(k, \theta)\sigma_{\hbar,V}(k, \theta)| \geq \frac{1}{2}, \quad \text{for all} \ \ |k|\geq M\,.\end{equation}
   Thus by the assumption on $\sigma_{\hbar,U}$ we get
   \[|\sigma_{\hbar,U}(k, \theta)| \geq \frac{1}{2|\sigma_{\hbar,V}(k, \theta)|} \geq \frac{1}{2C_{\hbar, V}}(1+|k|)^{\mu},,\]
 and we have proved that the symbol $\sigma_{\hbar,U}$
   is elliptic of order $\mu$.

   Conversely,  let us define
   \[\sigma_{\hbar,V_0}(k, \theta):= \frac{1}{\sigma_{\hbar,U}(k, \theta)}.\]
   By the analogous of \cite[Lemma 4.9.4]{ruzhansky2009pseudo} in the lattice $\hbar \Z^n$ setting, and assuming that $|k|>M$, for $M$ as in \eqref{M}, we get   $\sigma_{\hbar,V_0}\in S_{\rho, \delta}^{-\mu}(\hbar \Z^n\times \mathbb{T}^n)$. Hence by the composition formula
   \[\sigma_{\hbar,V_0 U} = \sigma_{\hbar,V_0}\sigma_{\hbar,U} - \sigma_{\hbar,T} \backsim 1 - \sigma_{\hbar,T},\]
   for some $T\in S_{\rho, \delta}^{-(\rho-\delta)}(\hbar \Z^n\times \mathbb{T}^n)$; that is $V_0 U = I -T$. The rest of the converse implication follows by the composition formula, see Theorem \ref{THM:comp} and a functional analytic argument as appears in the proof of \cite[Theorem 4.9.6]{ruzhansky2009pseudo}. It will then be omitted.
   
   Finally let us sketch the proof of the  formula \eqref{EQ:parexp}.
   We note that $I \backsim VU$ which implies  that $1  \backsim  \sigma_{\hbar,VU}(k, \theta).$ 
   An application of the composition formula as in Theorem \ref{THM:comp} yields
   \begin{eqnarray}
\label{elpt.par}
   1 &\backsim & \sum_{\gamma \geq 0}\frac{1}{\gamma !}\Big[D_{\hbar, \theta}^{(\gamma)}\sigma_{\hbar,V}(k, \theta)\Big]\Delta_{\hbar,k}^{\gamma} \sigma_{\hbar,U}(k, \theta)\nonumber \\
   &\backsim & \sum_{\gamma \geq 0}\frac{1}{\gamma !}\Big[D_{\hbar, \theta}^{(\gamma)} \sum_{j =0}^{\infty}\sigma_{\hbar,V_j} (k, \theta)\Big]\Delta_{\hbar,k}^{\gamma} \sum_{l = 0}^{\infty}\sigma_{\hbar,U_l}(k, \theta).
   \end{eqnarray}
   A combination of the formula \eqref{elpt.par} together with an argument similar to the one in the proof of \cite[Theorem 4.9.13]{ruzhansky2009pseudo} for the formula for the parametrix on $\T^n$, completes the proof. 
   \end{proof}
   \section{Link between toroidal and semi-classical  quantizations}
   \label{SEC:link}

   The toroidal quantization \cite{RT-JFAA,RT-Birk} gives rise to further developments and applications; see e.g. \cite{Paycha16,PZ14,ParZ14,Cardona14} to mention only a few. Therefore, it is important to stress out its link with the  semi-classical lattice  quantization that is exactly the topic of the current section. The idea behind this, is then to establish a way in which results on $\T^n$ can be transferred to $\hbar \Z^n$ and vice versa. Similar investigation has been performed in \cite{BKR20} in case of $\Z^n$, and here we verify that they still remain true after the addition of the semi-classical parameter $\hbar$. The importance of this link will be demonstrated in Section \ref{SEC:L2-HS}. Precisely, it provides us with a characterisation of compact operators on $\ell^2(\hbar \Z^n)$, see Corollary \ref{COR:comp}; the semi-classical version of Gohberg lemma, see Corollary \ref{COR:compG}; and conditions for the semi-classical operators in the Schatten-von Neumann classes, see  Theorem \ref{THM:sch}.
   
  \smallskip

   For $\tau_{\T^n}: \T^n \times \Z^n \rightarrow \mathbb{C}$, and for $v \in C^{\infty}(\T^n)$, recall the \textit{toroidal quantization}
    \begin{equation}\label{ClassicalOPT}
   	\Op_{\Tn}(\tau_{\T^n})u(\theta)=\sum_{k\in\Zn} e^{2\pi  i \theta\cdot k}\tau_{\T^n}(\theta,k) (\Ftn u)(k)\,.
   \end{equation}
   To distinguish between the toroidal and the semi-classical lattice quantization as in \eqref{neil}, we will denote them by $\Op_{\T^n}$ and $\Op_{\hbar \Z^n}$ (or $\Op_{\hbar}$), respectively. In both cases, we will use the notation $\overline{l},\overline{l},\cdots$ for elements of the lattice $\Z^n$, and the notation $\theta, \omega,\cdots$ for the elements of the torus $\T^n$.
    
   \begin{rem}[Relation between the toroidal and semi-classical Fourier transform]
  Recall that the toiroidal Fourier transform is defined by 
  \[
  \mathcal{F}_{\T^n}f(\overline{k})=\widehat{f}(\overline{k}):= \int_{\T^n}e^{-i2\pi \theta \cdot \overline{k}}f(\theta)\,\text{d}\theta\,,
  \]
  where $f \in C^{\infty}(\T^n)$ and  $\overline{k} \in \Z^n$.  In particular the operator $\mathcal{F}_{\T^n}$ is a bijection, with inverse $\mathcal{F}_{\T^n}^{-1}: \mathcal{S}(\Z^n) \rightarrow C^{\infty}(\T^n)$ given by
  \[
  (\mathcal{F}_{\T^n}^{-1}f)(\theta)=\sum_{k \in \Z^n}e^{-2\pi \theta \cdot k}f(\overline{k})\,.
  \]
  Recall now the semi-classical Fourier inversion formula as in \eqref{EQ:Finv}.  Observe that  for $\overline{k} \in \Z^n$:
  \begin{eqnarray}
      \label{RelationFtFz}
      \Ftn f(\overline{k}) & = & \int_{\T^n}e^{-2\pi i \theta \cdot \overline{k}} f(\theta) \text{d}\theta \nonumber\\
      & = & \int_{\T^n}e^{-2\pi \frac{i}{\hbar} \theta \cdot k} f(\theta) \text{d}\theta \quad (\text{where} \, k=\hbar \overline{k} \in \hbar \Z^n) \nonumber \\
      & = & (\mathcal{F}^{-1}_{\hbar \Z^n}f)(-k)\nonumber \\
      & = & (\mathcal{F}^{-1}_{\hbar \Z^n}f)(- \hbar \overline{k})\,.
  \end{eqnarray}
 \end{rem}
 From now on we will be using the notation $\Opzn$ to denote a semi-classical pseudo-differential operator, and the notation $\text{Op}_{\T}$ for the toroidal pseudo-differential operator in order to distinguish between the two.
 \smallskip 
 
 The next result allows us to reduce certain properties of semi-classical pseudo-differential operators to properties of toroidal pseudo-differential operators. 
 \begin{thm}\label{THM:link}
For a function $\sigma_\hbar:\hbar \Zn\times\Tn\to\C$ we define $\tau_\hbar : \Tn \times \Zn \rightarrow \mathbb{C}$ by $\tau_\hbar(\theta,\overline{k}):=\overline{\sigma_\hbar(-\hbar \overline{k},\theta)}$. Then we have the following relation 
\begin{equation}\label{EQ:link1}
\Opzn(\sigma_\hbar)=\mathcal{F}_{\hbar \Zn}^{-1}\circ \Optn(\tau_\hbar)^*\circ \mathcal{F}_{\hbar \Zn},
\end{equation} 
where $\Optn(\tau_\hbar)^*$ is the adjoint of the toroidal pseudo-differential operator $\Optn(\tau_\hbar)$ with symbol depending on the semi-classical parameter $\hbar$.
Moreover, we have
\begin{equation}\label{EQ:link2}
\Optn(\tau_\hbar)=\mathcal{F}_{\hbar \Zn}\circ \Opzn(\sigma_\hbar)^*\circ \mathcal{F}_{\hbar \Zn}^{-1},
\end{equation} 
where $\Opzn(\sigma_\hbar)^*$ is the adjoint of the semi-classical  pseudo-difference operator $\Opzn(\sigma_\hbar).$
\end{thm} 

\begin{proof}[Proof of Theorem \ref{THM:link}]
For $\varphi\in C^\infty(\Tn$) and for $\sigma_{\hbar}$ as in the hypothesis, consider the operator
$$
  T\varphi(k) := \int_{\mathbb{T}^n}e^{2\pi  \frac{i}{\hbar} k\cdot \theta}\sigma_\hbar(k,\theta)\varphi(\theta)\text{d}\theta\,,
$$
where $k=\hbar \overline{k} \in \hbar \Z^n$.
Formula \eqref{neil} rewritten in terms of the operator $T$ yields 
\begin{equation}\label{EQ:rel1}
\Opzn(\sigma_\hbar) = T\circ\mathcal{F}_{\hbar \Zn}.
\end{equation} 
The adjoint operator $T^*$ must satisfy the relation
\begin{equation}
\label{rel.1.l2}
 (T\varphi,\eta)_{\ell^2(\hbar \Zn)}=(\varphi,T^*\eta)_{L^2(\Tn)}\,, \quad \varphi, \eta \in \ell^2(\hbar \Z^n)\,.  \end{equation}
Now, expanding the left-hand side of \eqref{rel.1.l2} we get
\begin{multline*}
(T\varphi,\eta)_{\ell^2(\hbar \Zn)}=\sum_{k\in\hbar \Zn} T\varphi(k)\overline{\eta(k)}=
\sum_{k\in\hbar \Zn} \int_{\mathbb{T}^n}e^{2\pi  \frac{i}{\hbar} k\cdot \theta}\sigma_\hbar(k,\theta)\varphi(\theta)\overline{\eta(k)} \text{d}\theta\,,
\end{multline*} 
or equivalently
\[
 (T\varphi,\eta)_{\ell^2(\hbar \Zn)}= \int_{\mathbb{T}^n} \varphi(\theta)\left(\sum_{k\in\hbar \Zn}e^{2\pi  \frac{i}{\hbar} k\cdot \theta}\sigma_\hbar(k,\theta)\overline{\eta(k)}\right) \text{d}\theta\,.
\]
The latter means that
\begin{equation}\label{EQ:Tadj}
\begin{aligned}
T^*\eta(\theta) & =\sum_{k\in\hbar \Zn}e^{-2\pi  \frac{i}{\hbar} k\cdot \theta}\overline{\sigma_\hbar(k,\theta)}\eta(k) \\
& = \sum_{k\in\hbar \Zn}e^{2\pi i \overline{k}\cdot \theta}\overline{\sigma_\hbar(-\hbar \overline{k},\theta)}\eta(-k) \quad (\text{where}\,\,k=\hbar \overline{k}) \\
& = \sum_{k\in\hbar \Zn}e^{2\pi i \overline{k}\cdot \theta}\tau_\hbar(\theta,\overline{k}) \eta(-k)\\
& = \sum_{\overline{k} \in \Zn}e^{2\pi i \overline{k}\cdot \theta}\tau_{\hbar}(\theta,\overline{k})\eta(-\hbar \overline{k})\,.
\end{aligned}
\end{equation} 
Now since $\mathcal{F}_{\hbar \Zn}$ is a bijection, there exists $v$ such that $\mathcal{F}_{\hbar \Zn}^{-1}v(k)=\eta(k)$. Using the relation \eqref{RelationFtFz} the latter implies that $\Ftn v(-\overline{k})=\mathcal{F}_{\hbar \Zn}^{-1}v(k)$. Hence using the formula \eqref{ClassicalOPT} and equalities \eqref{EQ:Tadj} we can write
\begin{equation}\label{T*}
T^*\eta(\theta)=\sum_{k\in\hbar \Zn}e^{2\pi i \overline{k}\cdot \theta}\tau_\hbar(\theta,\overline{k}) \eta(-k)=\sum_{k\in \Zn}e^{2\pi i \overline{k}\cdot \theta}\tau_\hbar(\theta,\overline{k}) \Ftn v(\overline{k})=\Optn(\tau_\hbar) v(\theta)\,.
\end{equation}
On the other hand, we have 
\begin{eqnarray}
    \label{T**}
    \left(\Optn(\tau_\hbar) ({F}_{\hbar \Zn}\eta\right)(\theta) & = & \sum_{\overline{k}\in \Zn}e^{2\pi i \overline{k}\cdot \theta}\tau_\hbar(\theta,\overline{k}) \Ftn {F}_{\hbar \Zn}\eta(\overline{k})\nonumber\\
    & = & \sum_{\overline{k}\in \Zn}e^{2\pi i \overline{k}\cdot \theta}\tau_\hbar(\theta,\overline{k}) {\mathcal{F}}_{\hbar \Zn}^{-1}{\mathcal{F}}_{\hbar \Zn}\eta(-k)\nonumber\\
    & = & \sum_{\overline{k}\in \Zn}e^{2\pi i \overline{k}\cdot \theta}\tau_\hbar(\theta,\overline{k}) \eta(-k)\nonumber\\
    & = & \sum_{\overline{k}\in \Zn}e^{2\pi i \overline{k}\cdot \theta}\tau_\hbar(\theta,\overline{k}) \mathcal{F}_{\T^n}v(\overline{k})\nonumber\\
    & = & \Op_{\Tn}(\tau_\hbar)v(\theta)\,,
\end{eqnarray}
since by the above $\mathcal{F}_{\T^n}v(\overline{k})=\eta(-k)$. Now a combination of \eqref{T*} and \eqref{T**} yields

\begin{equation}\label{EQ:rel2}
T^*=\Optn(\tau_\hbar) \circ\mathcal{F}_{\hbar \Zn}.
\end{equation}
Consequently, the unitarity of the Fourier transform implies that
\begin{equation}\label{EQ:rel3}
T=\mathcal{F}_{\hbar \Zn}^{*}\circ \Optn(\tau_\hbar)^* = \mathcal{F}_{\hbar \Zn}^{-1}\circ \Optn(\tau_\hbar)^*\,.
\end{equation}
 Thus, combining \eqref{EQ:rel3} with \eqref{EQ:rel1}, we get 
\begin{equation}\label{Op.rel}
\Opzn(\sigma_\hbar)=\mathcal{F}_{\hbar \Zn}^{-1}\circ \Optn(\tau_\hbar)^*\circ \mathcal{F}_{\hbar \Zn}\,.
\end{equation}
Formula \eqref{EQ:link2} follows by \eqref{Op.rel} using similar arguments. This proof is  now complete.
\end{proof}
  
     \section{Applications}
   \label{SEC:L2-HS}
   In this section we investigate the conditions that can guarantee the boundedness of semi-classical pseudo-differential operators on $\ell^2(\hbar \Zn)$ and weighted $\ell^p(\hbar \Zn)$ spaces. Additionally the conditions for the membership in the Scahhhten classes are studied, as well as  a condition for the pseudo-differential operators to be Hilbert-Schmidt.

   \subsection{Continuity of semi-classical pseudo-differential operators}

    In this subsection we show results on the boundedness of the semi-classical pseudo-diffrential operators on different $\ell^p(\hbar \mathbb{Z}^n)$ spaces. In particular, Proposition \ref{PROP:HS} gives a sufficient and necessary condition on a semi-classical symbol $\sigma_\hbar$ for the the corresponding pseudo-differential operator to be Hilbert-Schmidt. Curiously, it also gives a sufficient condition for the operator to be bounded from $\ell^p(\hbar \mathbb{Z}^n)$ to $\ell^q(\hbar \mathbb{Z}^n)$, where $(p,q)$ are conjugate exponents. Regarding the special case where $p=q=2$ the sufficient condition for the boundedness of $\text{Op}_{\sigma_\hbar}$ becomes significantly more relaxed, in the sense that finitely many derivatives have to be bounded; see Theorem \ref{THM:L2}.
    
\smallskip

   Before moving on to prove our main results, let us recall that a bounded operator $T: H \rightarrow H$ acting on a Hilbert space $H$ is called \textit{Hilbert-Schmidt}, and we write $T\in \mathscr{L}(H)$, if it has finite Hilbert-Schmidt norm, i.e., if
    \[\|T\|_{\HS}^{2} := \sum_{i \in \mathcal{I}}\|Te_i\|_{H}^{2} < \infty\,,\]
    where $\{e_i : i \in \mathcal{I}\}$ is an orthonormal basis of $H$.
    
\smallskip

   \begin{proposition}\label{PROP:HS}
   The semi-classical pseudo-diffrential operator  $\Op_{\hbar}(\sigma_\hbar): \ell^{2}({\hbar\mathbb Z}^n)\rightarrow \ell ^{2}({ \hbar\mathbb Z}^n)$ is a Hilbert-Schmidt operator if and only if $\sigma_\hbar \in L^{2}(\hbar \mathbb{Z}^{n}\times\mathbb{T}^n)$. In this case, the Hilbert-Schmidt norm is given by 
   \begin{equation}\label{EQ:HS}
\|\Op_{\hbar}(\sigma_\hbar)\|_{\HS} = \|\sigma_\hbar\|_{L^{2}(\hbar\mathbb{Z}^{n}\times\mathbb{T}^n)}
   = \left(\sum_{k\in  \mathbb  \hbar \Z^n}\int_{\mathbb{T}^n}|\sigma_\hbar(k,\theta)|^{2}\text{\rm d}\theta \right)^{\frac12}.
\end{equation} 
   Furthermore,  if $\sigma_\hbar\in L^2(\hbar \Z^n \times\Tn)$ then $\Op(\sigma_\hbar):\ell^p(\hbar \Z^n)\to\ell^q(\hbar \Z^n)$ is bounded for all $1\leq p\leq 2$ and $\frac1p+\frac1q=1$, and we get that
   \begin{equation}\label{EQ:lplq}
\|\Op_{\hbar}(\sigma_\hbar)\|_{\mathscr{L}(\ell^p(\hbar \Z^n)\to \ell^q(\hbar \Z^n))}\leq \|\sigma_\hbar\|_{L^{2}(\hbar \Z^{n}\times\mathbb{T}^n)}.
\end{equation} 

   \end{proposition}
   \begin{proof}[Proof of Proposition \ref{PROP:HS}]
       Our first claim on the Hilbert-Schmidt norm is evident if one takes into account the Plancherel formula in this setting. The boundedness result follows the lines of \cite[Proposition 5.1]{BKR20} and will be omitted.
   \end{proof}

   \begin{thm}\label{THM:L2}
 Let $\varkappa\in\mathbb N$ and $\varkappa>n/2$.
Assume that the symbol $\sigma_\hbar:\hbar \Zn\times\Tn\to\C$ satisfies
\begin{equation}\label{EQ:l2conds}
|D_{\hbar,\theta}^{(\alpha)}\sigma_\hbar( k,\theta)|\leq C, \quad\textrm{ for all }\; (k,\theta)\in \hbar \Zn\times\Tn,
\end{equation} 
for all $|\alpha|\leq \varkappa$. Then the semi-classical pseudo-differential operator $\Op(\sigma_\hbar)$ extends to a bounded operator on $\ell^2(\hbar \Zn)$.
\end{thm} 

\begin{proof}
Let $\sigma_\hbar$ be as in the hypothesis. Then the symbol $\tau_\hbar$ related to $\sigma_\hbar$ as in Theorem \ref{THM:link} gives rise to the bounded on $L^2(\T^n)$ pseudo-differential operator $\Op_{\T^n}(\tau_\hbar)$; see \cite[Theorem 4.8.1]{ruzhansky2009pseudo}. On the other hand taking in to account that the operator $\mathcal{F}_{\hbar \Z^n}$ is an isometry from $\ell^2(\hbar \Zn)$ to $L^2(\Tn)$, and the relation \eqref{EQ:link1} in Theorem \ref{THM:link} we see that  $\Op(\sigma_\hbar)\equiv \Opzn(\sigma_\hbar)$ is bounded on $\ell^2(\hbar \Zn)$ if and only if
$\Optn(\tau_\hbar)$ is bounded on $L^2(\Tn)$. This completes the proof of Theorem \ref{THM:L2}.
\end{proof}

\subsection{Compactness, Gohberg lemma, and Schatten-von Neumann classes}

In this section we study the compactness  of the semi-classical pseudo-differential operators on $\ell^2(\hbar \Z^n)$, the distance between them and the space of compact operators on $\ell^2(\hbar \Z^n)$, and the sufficient conditions for the Schatten classes of semi-classical pseudo-differential operators; see Corollary \ref{COR:comp}, \ref{COR:compG} and Theorem \ref{THM:sch}, respectively. 

\smallskip

To ensure a self-contained presentation of our results, in the following remark we recall the necessary notions that are involved in the subsequent analysis.

\begin{rem} Let us recall some useful notions:
    \begin{enumerate}
        \item (Essential spectrum) Let $T$ be a closed linear operator on a complex Hilbert space $H$. The \textit{essential spectrum} of $T$, usually denoted by $\Sigma_{ess}(T)$ is the set of complex numbers $\lambda \in \mathbb{C}$ such that 
        \[
        T-\lambda I
        \]
        is  not a Fredholm operator, where $I$ is the identity operator.
        \item (Schatten-von Neumann
classes) Let $T: H \rightarrow H$ be a compact (linear) operator,  let $|T|:=(T^{*}T)^{1/2}$ be the \textit{absolute value of $T$}, and let let $s_n(T)$ be the \textit{singular values of $T$}, i.e., the eigenvalues of $|T|$. We say that the operator $T$ belongs to the \textit{Schatten-von Neumann class of operators $S_{p}(H)$}, where $1\leq p< \infty$, if
\[
\|T\|_{S_p}:=\left( \sum_{k=1}^{\infty}(s_{k}(T))^{p} \right)^{\frac{1}{p}}<\infty\,.
\]
The space $S_{p}$ is a Banach space if endowed with the natural norm $\|\cdot\|_{S_p}$
\item (Trace class operators) The Banach space $S_{1}(H)$ is the space of \textit{trace-class operators}, while for $T\in S_1$ the quantity
\[
\textnormal{Tr}(T):=\sum_{n=1}^{\infty}(Te_n,e_n)\,,
\]
where $(e_n)$ is an orthonormal basis in $H$, is well-defined and shall be called the \textit{trace $\textnormal{Tr}(T)$ of $T$}.
    \end{enumerate}
\end{rem}
In the sequel we have defined by $d$ the following quantity:
\begin{equation}\label{EQ:dd}
d:=\limsup_{ |k|\to\infty} \sup_{\theta\in\Tn} |\sigma_\hbar(k,\theta)|\,,
\end{equation}
where $(k, \theta) \in \hbar \Zn \times \Tn$.
Let us now present the main results of this subsection:

\begin{cor}\label{COR:comp}
Let $\sigma_\hbar\in S^0(\hbar \Zn\times\Tn)$. 
Then the semi-classical pseudo-differential  operator $\Opzn(\sigma_\hbar)$ is compact on $\ell^2(\hbar \Zn)$ if and only if $d=0$, where $d$ is as in \eqref{EQ:dd}.
Moreover, we have
$$
\Sigma_{ess}(\Opzn(\sigma_\hbar)) \subset \{\lambda\in\mathbb{C}:|\lambda|\leq d\}.
$$
\end{cor} 
\begin{proof}[Proof of Corollary \ref{COR:comp}]
The main idea is a combination, on the one hand, of the fact that  the compactness, Fredhomlness, and the index are invariant under the action of unitary and of the relation \eqref{EQ:link1} in Theorem \ref{THM:link}, and on the other hand of \cite[Theorem 3.2]{DR:Gohberg} on the toroidal pseudo-differential operators. The rest of the arguments follows the lines of \cite[Corollary 5.3]{BKR20} and are omitted. 
\end{proof}


The next result gives a lower bound for the distance between a given operator and the space of compact operators on $\ell^2(\hbar \Zn)$. Such type of  statements were first shown by Gohberg in \cite{Gohberg}, and are now bearing his name. We refer to \cite{Mol11,Pir11} for such a result on the circle $\mathbb T^1$, and to \cite{DR:Gohberg} on general compact Lie groups. The analogous result in the lattice case was given in \cite{BKR20}.

\begin{cor}\label{COR:compG}{Gohberg lemma}
Let $\sigma_\hbar:\hbar \Zn\times\Tn\to\C$ be such that
\begin{equation}\label{EQ:Gcond}
|\sigma_\hbar( k,\theta)|\leq C,\quad |\nabla_{\hbar, \theta}\sigma_\hbar( k,\theta)|\leq C,\quad |\Delta_{\hbar,q}\sigma_\hbar(k,\theta)|\leq C(1+|k|)^{-\rho},
\end{equation} 
 for some $\rho>0$ and for all $q\in C^\infty(\Tn)$ with $q(0)=0$ and all $(k,\theta)\in\hbar \Zn\times\Tn$.
Then for all compact operators $K$ on $\ell^2(\hbar \Zn)$ we have
$$
\|\Opzn(\sigma_\hbar)-K\|_{\mathcal{L}(\ell^2(\hbar \Zn))}\geq d.
$$
In particular, this conclusion holds for any $\sigma_\hbar\in S^0(\hbar \Zn\times\Tn).$
\end{cor} 
\begin{proof}
    The proof is a consequence of \eqref{EQ:link1} in Theorem \ref{THM:link} and \cite[Theorem 3.]{DR-JMPA}. See also \cite[Corollary 5.4]{BKR20}.
\end{proof}

The next theorem is an application of the developed calculus that presents the conditions ensuring that the corresponding operators belong to Schatten classes. 

\begin{thm}\label{THM:sch}
Let $0<p \leq 2$. We have the following implication
\begin{equation}\label{p=2}
\sum_{k \in \hbar \Z^n}\|\sigma_{\hbar}( k, \cdot)\|^{p}_{L^2(\T^n)}<\infty \Longrightarrow \text{Op}_{\hbar \Z^n}(\sigma_\hbar) \quad \text{is $p$-Schatten operator on} \quad \ell^2(\hbar \Z^n)\,.
\end{equation}
In particular, if the left-hand side of \eqref{p=2} holds true for $p=1$, then the operator $\text{Op}_{\hbar \Z^n}(\sigma_\hbar)$ is trace class, and its trace can be calculated as follows:
\begin{equation}\label{trace.thm}
\textnormal{Tr}(\Opzn(\sigma_\hbar))=\sum_{k \in \Z^n}\int_{\T^n} \sigma_{\hbar} (k, \theta)\,d\theta=\sum_{j\in \mathcal{J}}\lambda_j\,,
\end{equation}
where the set $\{\lambda_j, j \in \mathcal{J}\}$ is the set of eigenvalues of $\text{Op}_{\hbar \Z^n}(\sigma_\hbar)$ (multiplicities counted).
    \end{thm}
\begin{proof}[Proof of Theorem \ref{THM:sch}]
    For the case $p \in (0,1]$ the notion of $p$-nuclearity\footnote{The notion of $p$-nuclearity initiated by Grothedieck in \cite{Groth-MAMS}. We refer to the work \cite{DR-JMPA} for a detailed discussion of it.} and $p$-Shatten classes coincide; see \cite{Oloff72} and \cite[Section 6.3.2.11]{Pietsch-history}. Taking this into account, the result for the case $p \in (0,1]$ follows as a consequence of \eqref{EQ:link1} and \cite[Corollarry 3.12]{DR:Gohberg}. For the trace class operators, using the expression \eqref{EQ:kerdiag} for the kernel, one can prove the first equality in \eqref{trace.thm}. The second equality in \eqref{trace.thm} is the well-known Lidskii formula \cite{Lidskii59}. For the case where $p \in [1,2]$ the result follows by interpolation using \eqref{EQ:HS}. This completes the proof.
\end{proof}

\begin{rem}
In \cite{Mantoiu-Ruzhansky-DM} the authors proved a result on the analysis of the Schatten classes of operators on  locally compact separable unimodular groups of Type I, which in our case reads as follows: for $p \in [2,\infty)$ and for $p'$ the  conjugate exponent of $p$ (i.e. $\frac{1}{p}+\frac{1}{p'}=1$) we have 
\[
\sum_{k \in \hbar \Z^n}\|\sigma_{\hbar}( k, \cdot)\|^{p'}_{L^{p'}(\T^n)}<\infty \Longrightarrow \text{Op}_{\hbar \Z^n}(\sigma_\hbar) \quad \text{is $p$-Schatten operator on} \quad \ell^2(\hbar \Z^n)\,.
\]
\end{rem}

\subsection{Weighted $\ell^2$-boundedness}
In this subsection, we present a result on the boundedness of semi-classical pseudo-differential operators on weighted $\ell^2(\hbar \Zn)$ spaces defined  below:
\begin{defn}{(Weighted $\ell^p_s(\hbar \Zn)$ space)}
 For $s\in \R$ and $1\leq p<\infty$ we define the \textit{weighted space $\ell^p_s(\hbar \Zn)$} as the space of all $f:\hbar \Zn\to\C$ such that 
   \begin{equation}\label{EQ:l2s}
 \|f\|_{\ell^p_s(\hbar \Z^n)}:=\left(\sum_{k\in \hbar \Z^n} (1+|k|)^{sp} |f( k)|^p\right)^{1/p}<\infty.
\end{equation} 
\end{defn}
It is easy to check that the symbol $a_{\hbar,s}(k)=(1+|k|)^{s}$ belongs to the semi-classical class of symbols $S^s_{1,0}(\hbar \Zn\times\Tn)$, while also that
  \[
  f\in \ell^p_s(\hbar \Z^n)\quad \text{if and only if}\quad  \Op(a_s) f\in \ell^p(\hbar \Z^n)\,.
  \]
  The latter observation gives rise to the following identification:
   \begin{equation}\label{EQ:l2s}
 \ell^p_s(\hbar \Zn)=\Op(a_{-s})(\ell^p(\hbar \Zn)).
\end{equation}

     \begin{cor}\label{COR:L2}
   Let $r\in\R$ and let $\sigma_\hbar\in S^r_{0,0}(\hbar \Zn\times\Tn)$. Then, the semi-classical pseudo-differential operator  $\Op(\sigma_\hbar)$ is a bounded  from  
   $\ell^2_{s}(\hbar \Zn)$ to $\ell^2_{s-r}(\hbar \Zn)$ for all $s\in\R$.
      \end{cor}
      \begin{proof}
  If $T=\Op(\sigma_\hbar)\in \Op(S^r_{0,0}(\hbar \Zn\times\Tn))$, then by using the composition formula as in Theorem \ref{THM:comp} we have 
  \[
  P=\Op(a_{\hbar,s-r})\circ T\circ \Op(a_{\hbar,-s}) \in \Op(S^0_{0,0}(\hbar \Zn\times\Tn))\footnote{We note that in this case, and since $a_{\hbar,s} \in S^s_{0,1}(\hbar \Zn \times \Tn)$ the asymptotic formula as in Theorem \ref{THM:comp} is well-defined.}
  \]
  
and the operator $P$ is by Theorem \ref{THM:L2} bounded on $\ell^2(\hbar \Zn)$.
  We can write
  \[
  Tf=\Op(a_{\hbar,r-s})\circ P \circ \Op(a_{\hbar,s})f\,,
  \]
  where $P$ is as above. Now, if $f\in \ell^2_{s}(\hbar \Zn)$, then, since $\Op(a_{\hbar,s})f, (P \circ \Op(a_{\hbar,s}))f \in \ell^2(\hbar \Zn)$, we also get that $Tf\in \Op(a_{\mu-s}) \ell^2(\hbar \Zn)$. The proof is now complete in view of the identification \eqref{EQ:l2s}.
\end{proof}

 \subsection{G{\aa}rding and sharp G{\aa}rding inequalities on $\hbar \mathbb Z^n$}
 
 \medskip
 
Let us recall the following result on the torus $\Tn$ as in \cite[Corollary 6.2]{Ruzhansky-Wirth:functional-calculus}:  
\begin{cor}{(G{\aa}rding inequality on $\Tn$)}\label{A}
Let $0\leq \delta < \rho \leq 1$ and $m > 0$. Let $ B \in \Op _{\mathbb{T}^n}{S}^{2m}_{\rho,\delta}(\mathbb{T}^n \times \Z^n)$ be an  elliptic toroidal pseudo-differential operator  such that $\sigma_B (\theta,\overline{k}) \geq  0$, for all $\theta \in \Tn$ and co-finitely many $\overline{k}\in \Zn$.
Then there exist  $C_{0},C_{1} > 0$ such that for all $f \in H^m(\mathbb{T}^n)$ we have
\begin{equation*}
{\rm Re} (Bf,f)_{L^2(\mathbb{T}^n)} \geq C_{0}||f||^{2}_{H^m(\mathbb{T}^n)}-C_{1}||f||^{2}_{L^2(\mathbb{T}^n)}.
\end{equation*}
\end{cor}
Let us now show the semi-classical analogue of  G{\aa}rding inequality on  $\hbar \Z^n$. As there is no regularity concept on the lattice, the statement is given in terms of weighted $\ell^2(\hbar \mathbb{Z}^n)$-spaces.

\begin{thm}[G{\aa}rding inequality on $\hbar \Z^n$]\label{THM:Garding}
Let $0\leq \delta < \rho \leq 1$ and $m > 0$. Let $ P \in \Op _{\hbar \Z^n} {S}^{2m}_{\rho,\delta}(\hbar \Z^n \times \mathbb{T}^n)$ be an elliptic semi-classical pseudi-differential operator  such that $\sigma_{\hbar,P} (k,\theta) \geq  0$ for all  $\theta$ and for co-finitely many $k \in \hbar \Zn$.
Then there exist  $C_{1},C_{2} > 0$ such that for all $g \in \ell^{2}_m(\mathbb{T}^n)$  we have
\begin{equation}\label{gard 1}
{\rm Re} (Pg,g)_{\ell^2(\hbar \Z^n)} \geq C_{0}||g||^{2}_{\ell^{2}_m(\hbar \Z^n)}-C_{1}||g||^{2}_{\ell^2(\hbar \Z^n)}.
\end{equation}
\end{thm}
\begin{proof}
Let us define $\tau_\hbar(\theta,\overline{k}) =  \overline{\sigma_{\hbar,P}(- k,\theta) } $, where $k=\hbar \overline{k}$.
Then, using  Theorem \ref{THM:link} we have
\begin{equation}\label{gard 2}
P = \Op_{\hbar \Z^n}(\sigma_{\hbar,P}) = \mathcal{F}^{-1}_{\hbar \Z^n} \circ \Op _{\mathbb{T}^n}(\tau_\hbar)^* \circ \mathcal{F}_{\hbar \Z^n}\,.
\end{equation} 
The latter implies that if  $\sigma_{\hbar,P}\geq 0$ is elliptic on $\mathbb{T}^n$,  then also $\tau_\hbar\geq 0 $ is elliptic on  $\hbar \Z^n$. Hence, using the G{\aa}rding inequality on $\Tn$, see Corollary \ref{A}, we get that for all $f \in H^m(\mathbb{T}^n)$ 
\begin{eqnarray}\label{gard 3}
    {\rm Re}( \Op_{\mathbb{T}^n}(\tau_\hbar)^* f,f)_{L^2(\mathbb{T}^n)} & = & {\rm Re}( \Op_{\mathbb{T}^n}(\tau_\hbar)f,f)_{L^2(\mathbb{T}^n)} \nonumber \\
    & \geq & C_{0}||f||^{2}_{H^{m}(\mathbb{T}^n)}-C_{1}||f||^{2}_{L^2(\mathbb{T}^n)} \nonumber \\
    & = & C_{0}||g||^{2}_{\ell^{2}_m(\hbar \Z^n)}-C_{1}||g||^{2}_{\ell^2(\hbar \Z^n)}\,,
\end{eqnarray}
 where $f$ is such that $f = \mathcal{F}_{\hbar \Z^n}g$, so that
\begin{equation*}
\| g\|_{ H^m (\mathbb{T}^n)} = \| g\|_{\ell^{2}_m(\hbar \Z^n)} \qquad \text{and} \qquad \| g\|_{ L^2 (\mathbb{T}^n)} = \|g\|_{\ell^{2}(\hbar \Z^n)}.
\end{equation*} 
Now, by \eqref{gard 2} we can write
\begin{equation}\label{gard 41}
Pg = \mathcal{F}^{-1}_{\hbar \Z^n} \circ \Op _{\mathbb{T}^n}(\tau_\hbar)^* \circ \mathcal{F}_{\hbar \Z^n} g =  \mathcal{F}^{-1}_{\hbar \Z^n} \circ \Op _{\mathbb{T}^n}(\tau_\hbar)^* f,
\end{equation} 
so that $ \mathcal{F}_{\hbar \Z^n}Pg =   \Op _{\mathbb{T}^n}(\tau_\hbar)^* f $. 
Thus, by \eqref{gard 3} we have
\[
{\rm Re}( \mathcal{F}_{\hbar \Z^n} Pg, \mathcal{F}_{\hbar \Zn}g)_{L^2(\Tn)} \geq  C_{0}||f||^{2}_{\ell^{2}_m(\hbar \Z^n)}-C_{1}||f||^{2}_{\ell^2(\hbar \Z^n)}\,,
\]
where the last can be rewritten as
\[
{\rm Re}(\mathcal{F}_{\hbar \Z^n}^* \mathcal{F}_{\hbar \Z^n} Pg,  g)_{L^2(\mathbb{T}^n)} \geq  C_{0}||f||^{2}_{\ell^{2}_m(\hbar \Z^n)}-C_{1}||f||^{2}_{\ell^2(\hbar \Z^n)}\,.
\]
Since $\mathcal{F}_{\hbar \Z^n}^* \mathcal{F}_{\hbar \Z^n} = Id$, we obtain
\[
{\rm Re}( Pg,g)_{\ell^2(\hbar \Z^n)}  \geq    C_{0}||g||^{2}_{\ell^{2}_m(\hbar \Z^n)}-C_{1}||g||^{2}_{\ell^2(\hbar \Z^n)}\,.
\]
This completes the proof of Theorem \ref{THM:Garding}.
\end{proof}
Next we show the sharp G{\aa}rding inequality on $\hbar \Z^n$. Before doing so, let us recall how the sharp G{\aa}rding inequality on compact Lie groups, see \cite[Theorem 2.1]{ruzhansky2011sharp}, reads in the case of the torus $\Tn$.

\begin{thm}[Sharp G{\aa}rding inequality on $\mathbb{T}^n$]\label{THM:sG}
Let $B \in \Op_{\mathbb{T}^n}{S}^m(\mathbb{T}^n \times  \Z^n)$ be a toroidal pseudo-differential operator with  symbol ${\sigma_\hbar}(\theta,\overline{k}) \geq 0$ for all $(\theta,\overline{k}) \in \mathbb{T}^n \times \Z^n$. Then there exists $C < \infty$ such that
\[{\rm Re}( Bg,g)_{L^2(\mathbb{T}^n)} \geq -C\|g\|_{H^{\frac{m-1}{2}}(\mathbb{T}^n)},\]
for all $g \in H^{\frac{m-1}{2}}(\mathbb{T}^n)$.
\end{thm}

Let us know prove the analogous result in the semi-classical setting $\hbar \Zn$.
\begin{thm}[Sharp G\"{a}rding inequality on $\hbar \Z^n$]\label{THM:Sgard}
Let $P \in \Op_{\hbar \Z^n}{S}^m(\hbar \Z^n \times \mathbb{T}^n)$ be a semi-classical pseudo-diffrential operator with symbol $\sigma_{\hbar,P}( k,\theta) \geq 0$ for all $(k,\theta) \in  \hbar \Z^n \times \mathbb{T}^n$. Then there exists $C < \infty$ such that  
\[{\rm Re}( Pg,g)_{\ell^2(\hbar \Z^n)} \geq -C\|g\|_{\ell^{2}_{\frac{m-1}{2}}(\hbar \Z^n)}\]
for all $g \in \ell^{2}_{\frac{m-1}{2}}(\hbar \Z^n)$.
\end{thm}
\begin{proof}
Let $\tau_\hbar(\theta,\overline{k}) =  \overline{\sigma_{\hbar,P}(- k,\theta) } $, where $k=\hbar \overline{k}$.
Using Theorem \ref{THM:link} we can write 
\begin{equation}\label{gard 2}
P = \Op_{\hbar \Z^n}(\sigma_{\hbar,P}) = \mathcal{F}^{-1}_{\hbar \Z^n} \circ \Op _{\mathbb{T}^n}(\tau_\hbar)^* \circ \mathcal{F}_{\hbar \Z^n}.
\end{equation}
Following the lines of Theorem \ref{THM:Garding} and using the Sharp G{\aa}rding inequality on $\mathbb{T}^n$ we get 
\begin{eqnarray*}
{\rm Re}(Pg,g)_{\ell^2(\hbar \Z^n)} &=&{\rm Re}( \mathcal{F}^{-1}_{\hbar \Z^n}  \Op_{\mathbb{T}^n}(\tau_\hbar)^* f, \mathcal{F}^{-1}_{\hbar \Z^n} f)_{\ell^2(\hbar \Z^n)}\\
&=& {\rm Re}( \Op_{\mathbb{T}^n}(\tau_\hbar)^* f,f)_{L^2(\mathbb{T}^n)}\\
&=& {\rm Re}( \Op_{\mathbb{T}^n}(\tau_\hbar) f,f)_{L^2(\mathbb{T}^n)}  \\
&\geq & -C\|f\|_{H^{\frac{m-1}{2}}(\mathbb{T}^n)}\\
&=& -C\|g\|_{\ell^{2}_{\frac{m-1}{2}}(\hbar \Z^n)}.
\end{eqnarray*}
The proof of Theorem \ref{THM:Sgard} is now complete.
\end{proof}

 \subsection{Existence and uniqueness of the solutions to parabolic equations on $\hbar \Zn$}
 
In this subsection we will apply the G{\aa}rding inequalities in our semi-classical setting to prove the well-posedeness of the classical parabolic equation
\begin{equation}\label{42}
\begin{cases}
\frac{\partial w}{\partial t} - Dw &= g, \qquad t\in [0,T],\quad T>0,\\
w(0)  &= w_0\,,
\end{cases}
\end{equation}  
where in this case the classical differential operator is replaced by a semi-classical pseudo-differential operator denoted by $D$ and has a symbol in the class ${S}^{m}_{1,0}(\hbar \Z^n \times \mathbb{T}^n)$.
\begin{thm}\label{THM:parabolic}
Let $r>0$ and $D \in \Op_{\hbar \Z^n}{S}^{r}_{1,0}(\hbar \Z^n \times \mathbb{T}^n)$ be a semi-classical pseudo-differential operator.
Assume also that  there exist $C_{0} >0$ and  $R>0$ such that for all $\theta\in\mathbb T^n$, we have
\begin{equation}\label{gard 5}
 -\sigma_{\hbar,D}( k,\theta) \geq C_{0}|k|^r  \qquad \text{for} \ |k| \geq R.
 \end{equation}
If for $w_0$ and $g$ as in \eqref{42}, we have $w_0\in \ell^2(\hbar \Z^n)$ and $g \in L^1([0,T],\ell^2(\hbar \Z^n))$, then the equation \eqref{42} has a unique solution $w \in C([0,T],\ell^2(\hbar \Z^n))$ that satisfies the estimate
\begin{equation}\label{gard 6}
\|u(t)\|^{2}_{\ell^2(\hbar \Z^n)} \leq C\Big(\|u_0\|^{2}_{\ell^2(\hbar \Z^n)} +  \int^{t}_{0} \|f(s)\|^{2}_{\ell^2(\hbar \Z^n)} d s\Big)\,,
 \end{equation}
for some $C>0$ and for all $t\in [0,T]$.

\end{thm}
\begin{proof}
Let $w$ be the solution to the equation \eqref{42}. If $\sigma_{\hbar,D}^*( k,\theta)$ stands for the symbol of the adjoint operator $D^*$, then by condition \eqref{gard 5} there exists $C'_{0 }>0$ such that
\begin{equation*}
 |-(\sigma_{\hbar,D} + \sigma_{\hbar,D}^*)(k,\theta)| \geq C'_{0 }|k|^r  \qquad \text{for} \ |k| \geq R.
 \end{equation*}
 Then, using the the G{\aa}rding inequality as in Theorem \ref{THM:Garding} and Theorem \ref{THM:adjoint} on the adjoint operators, we get 
 \begin{equation}\label{gard 7}
 -\Big( (D+D^*)w,w\Big)_{\ell^2(\hbar \Z^n)} \geq C_{1}\|w\|^{2}_{\ell^{2}_{\frac{r}{2}}(\hbar \Z^n)}-C_{2}\|w\|^{2}_{\ell^2(\hbar \Z^n)}.
 \end{equation}
 
 On the other hand we have 
 \begin{eqnarray}\label{withD}
 \frac{\partial }{\partial t} \|w\|^{2}_{\ell^2(\hbar \Z^n)} &=&  \frac{\partial }{\partial t}\bigg(w(t),w(t)\bigg)_{\ell^2(\hbar \Z^n)}=\bigg( \frac{\partial w}{\partial t},w\bigg)_{\ell^2(\hbar \Z^n)} + \bigg(w, \frac{\partial w}{\partial t}\bigg)_{\ell^2(\hbar \Z^n)}\nonumber\\
&=& \bigg(Dw +g,w \bigg)_{\ell^2(\hbar \Z^n)} + \bigg(w,Dw +g \bigg)_{\ell^2(\hbar \Z^n)}\nonumber\\
&=& \bigg((D+D^*)w ,w \bigg)_{\ell^2(\hbar \Z^n)} + 2{\rm Re}(w,g)_{\ell^2(\hbar \Z^n)}. 
\end{eqnarray}
Hence a combination of \eqref{gard 7} together with \eqref{withD} gives
\begin{eqnarray*}
   \frac{\partial }{\partial t} \|w\|^{2}_{\ell^2(\hbar \Z^n)} & \leq &  -C_{1}\|w(t)\|^{2}_{\ell^{2}_{\frac{r}{2}}(\hbar \Z^n)}+C_{2}\|w(t)\|^{2}_{\ell^2(\hbar \Z^n)} + \|w(t)\|^{2}_{\ell^2(\hbar \Z^n)} + ||g||^{2}_{\ell^2(\hbar \Z^n)}\\
   & \leq & (C_{2} +1)\|w(t)\|^{2}_{\ell^2(\hbar \Z^n)} + \|g\|^{2}_{\ell^2(\hbar \Z^n)}\,.
\end{eqnarray*}
 An application of Gronwall's lemma to the latter gives 
 \begin{equation*}
\|w(t)\|^{2}_{\ell^2(\hbar \Z^n)} \leq C_{\hbar}\Big(\|w_0\|^{2}_{\ell^2(\hbar \Z^n)} +  \int^{T}_{0} \|f(s)\|^{2}_{\ell^2(\hbar \Z^n)} d s\Big)\,,
 \end{equation*}
 and we have proved \eqref{gard 6}.
 
 The existence of a solution  $w \in C_{\hbar}([0,T],\ell^2(\hbar \Z^n))$ to the equation \eqref{42} follows by a modification of the standard Picard's theorem.
 
 To prove the well-posedeness, let $w,v$ be two solutions of \eqref{42}. Then by setting $u:= w-v$, we have
\begin{equation*}
 \begin{cases}
\frac{\partial u}{\partial t} - Du &= 0, \; t\in [0,T],\\
u(0) &= 0.
\end{cases}
 \end{equation*} 
Now, the estimate \eqref{gard 6} implies that $\|w(t)\|_{\ell^2(\hbar \Z^n)} = 0$, which in turn gives that $w(t)=v(t)$ for all $t \in [0,T]$, completing the proof.
\end{proof}

\subsection{Boundedness and compactness on $\ell^p(\hbar \Zn)$}

   
   
   \ 
\ 
   In the result that follows we show the $\ell^p(\hbar \Zn)$-boundedness of a semi-classical pseudo-differential operator. Here the bound of the operator norm depends on the bound of the operator symbol which does not necessarily needs to be regular or obey a decay condition. An analogous result for pseudo-differential operators on the lattice $\Zn$ is established in \cite[Proposition 5.12]{BKR20}, while for the special case where $n=1$ in \cite{molahajloo2009pseudo}. 
   
   \begin{proposition}\label{compact} Let $1\leq p< \infty.$ 
   Let also $\sigma_\hbar:\hbar \Z^n\times \mathbb{T}^n\to\C$ be a measurable function such that
   \[|(\mathcal{F}_{\mathbb T^n}\sigma_\hbar)( k, m)| \leq C_|\lambda(m)|, \quad \textrm{ for all }\; k,m \in \hbar \Z^n,\]
   where $C>0$ is a positive constant, $\lambda$ is some function on $\hbar \Zn$  such that $\lambda\in \ell^{1}(\hbar \Z^n)$ and $\mathcal{F}_{\mathbb T^n}\sigma_\hbar$ is the Fourier transform of $\sigma_\hbar$ in the second variable.
   Then, $\Op_{\hbar \mathbb{Z}^n}(\sigma_\hbar): \ell^{p}(\hbar \Z^n)\rightarrow \ell^{p}(\hbar \Z^n)$ is a bounded linear operator and its norm is bounded from above. Particularly we have
   \[\|\Op(\sigma_\hbar)\|_{\mathscr{L}(\ell^{p}(\hbar \Z^n))} \leq C\|\lambda\|_{\ell^{1}(\hbar \Z^n)}. \]
    \end{proposition}
   
   
   \begin{proof}[Proof of Proposition \ref{compact}]
   For $g\in \ell ^{1}(\hbar \Z^n)$, and for $k, m \in \hbar \Zn$ we can write 
   \begin{eqnarray*}
   \text{Op}_{\hbar \mathbb{Z}^n}(\sigma_\hbar)g(k) &=& \sum_{m\in \hbar \Z^{n}}g(m)\int_{\mathbb{T}^n}e^{-2\pi  \frac{i}{\hbar}(m-k)\cdot \theta}\sigma_\hbar(k,\theta)\text{d}\theta \\
   &=& \sum_{m\in \hbar \Z^{n}}g(m)(\mathcal{F}_{\mathbb T^n}\sigma_\hbar)(k,m-k)\\
   & = &  \sum_{m\in \hbar \Z^{n}}g(m)(\mathcal{F}_{\mathbb T^n}\sigma_\hbar)^{\sim}(k,k-m) \\
   &=& ((\mathcal{F}_{\mathbb T^n}\sigma_\hbar)^{\sim}(k,\cdot) * g)(k)\,,
   \end{eqnarray*}
   where we have  defined 
   \[(\mathcal{F}_{\mathbb T^n}\sigma_\hbar)^{\sim}(k,m) = :(\mathcal{F}_{\mathbb T^n}\sigma_\hbar)(k,-m). \]
   From the above we can estimate
   \begin{equation}\label{conv}
    \|\text{Op}_{\hbar \mathbb{Z}^n}(\sigma_\hbar)g\|_{\ell^{p}(\hbar \Z^n)}^{p}=\sum_{k\in \hbar \Z^{n}}|((\mathcal{F}_{\mathbb T^n}\sigma_\hbar)^{\sim}(k,\cdot) * g)(k) |^{p} \leq \sum_{k\in \hbar \Z^{n}}((|(\mathcal{F}_{\mathbb T^n}\sigma_\hbar)^{\sim}(k,\cdot)| * |g|)(k)) ^{p}.
   \end{equation}
   Taking into account the assumption on $\sigma_{\hbar}$, an application of Young's inequality for convolution yields
   \begin{equation}
       \label{app.YI}
       \sum_{k\in \hbar \Z^{n}}((|(\mathcal{F}_{\mathbb T^n}\sigma_\hbar)^{\sim}(k,\cdot)| * |g|)(k)) ^{p} \leq  C^p \sum_{k\in \hbar \Z^{n}}\Big((|\lambda| * |g|)(k)\Big)^{p} \leq C^p \|\lambda\|_{\ell ^{1}(\hbar \Z^n)}^{p}\|g\|_{\ell^{p}(\hbar \Z^n)}^{p}\,.
   \end{equation}
    The latter combined with the density of $\ell^{1}(\hbar \Z^n)$ in $\ell^{p}(\hbar \Z^n)$, where $1\leq  p< \infty$, completes the proof.
   \end{proof}
  
In the next result we strengthen the assumption on the symbol $\sigma_\hbar$ to guarantee that the corresponding semi-classical pseudo-difference operator $\Op_{\hbar \mathbb{Z}^n}(\sigma_\hbar):\ell^{p}(\hbar \Z^n)\rightarrow \ell^{p}(\hbar \Z^n)$ is bounded but also compact.

   \begin{thm}\label{THM:cpt}
   Let $\sigma_\hbar$ and  $\lambda$ be as in the hypothesis of Proposition \ref{compact}. Let also $\omega$ be a positive function on $\hbar \Zn$. Suppose also that $\sigma_\hbar$ satisfies
      \[|(\mathcal{F}_{\mathbb{T}^n}\sigma_\hbar)(k,m)| \leq \omega(k)|\lambda(m)|, \quad\textrm{ for all }\; m, k \in \hbar \Z^n,\]
   where
   \[\lim\limits_{|k|\rightarrow \infty}\omega(k) = 0.\]
   Then the pseudo-difference operator $\Op_{\hbar \mathbb{Z}^n}(\sigma_\hbar): \ell^p(\hbar \Z^n)\rightarrow \ell ^p(\hbar \Z^n)$ is a compact operator for all $1\leq p<\infty.$
   \end{thm}
   \begin{proof}
   We will show that $\Op_{\hbar \mathbb{Z}^n}(\sigma_\hbar)$ is the limit (in the operator norm sense) of a sequence of compact operators $\Op_{\hbar \mathbb{Z}^n}(\sigma_{\hbar,n})$ on $\ell^p(\hbar \Zn)$. To this end, we define the sequence of the corresponding symbols $\sigma_{\hbar,n}$:
    \[\sigma_{\hbar,n}(k,\theta) := \left\{\begin{aligned}
   \sigma_\hbar(k,\theta) , \quad |k|\leq n,\\
   0, \quad |k|> n. \\
   \end{aligned}\right.\]
   For $g \in \ell^1(\hbar \Zn)$ we have 
   \begin{equation}\label{sub}
\begin{aligned}
\big(\text{Op}_{\hbar \Z}(\sigma_\hbar) - \text{Op}_{\hbar \Z}(\sigma_{\hbar,N})\big)g(k) & = \int_{\mathbb T^n}e^{2\pi  \frac{i}{\hbar}k\cdot \theta}(\sigma_\hbar - \sigma_{\hbar,n})(k,\theta)\widehat{g}(\theta)\text{d}\theta \\
&  = \sum_{m\in  \hbar \Z^n}g(m)\int_{\mathbb T^n}e^{-2\pi  \frac{i}{\hbar}(m-k)\cdot \theta}(\sigma_\hbar - \sigma_{\hbar,n})(k,\theta)\text{d}\theta \\
&   = \sum_{m\in \hbar \Z^n}g(m)(\mathcal{F}_{\mathbb T^n}(\sigma_\hbar - \sigma_{\hbar,n}))(k,m-k).
\end{aligned}
\end{equation} 
  Then arguing as we did in Proposition \ref{compact} we  get
   \begin{eqnarray*}
   \|\big(\text{Op}_{\hbar \Z}(\sigma_\hbar) - \text{Op}_{\hbar \Z}(\sigma_{\hbar,n})\big)g\|_{\ell^p(\hbar \Z^n)}^{p}
   & \leq &  \sum_{k\in \hbar \Z^n}\Bigg(\Big(\big|\big(\mathcal{F}_{\mathbb T^n}(\sigma_\hbar - \sigma_{\hbar,n})\big)^{\sim}(k,\cdot)\big|\ast\big| g\big|\Big)(k)\Bigg)^{p}\\ 
   &\leq & \sum_{|k|>n}\Bigg(\Big(\big|\big(\mathcal{F}_{\mathbb T^n}\sigma_\hbar\big)^{\sim}(k,\cdot)\big|\ast\big| g\big|\Big)(k)\Bigg)^{p}\\
    &\leq & \sum_{|\overline{k}|>n_0}\Bigg(\Big(\varepsilon\big|\lambda\big|\ast\big| g\big|\Big)(k)\Bigg)^{p}\,,
   \end{eqnarray*}
   where for the estimate in the last line we have used the condition on the symbol $\sigma_{\hbar}$ as in the hypothesis, together with the fact that for $\omega$ as in the hypothesis and for every $\varepsilon$, there exists $n_0$ such that $|\omega(k)|<\varepsilon$ for all $k>n_0$.
 Now, an application of Young's inequality to the latter gives
   \begin{eqnarray*}
   \|\big(\text{Op}_{\hbar \Z}(\sigma_\hbar) - \text{Op}_{\hbar \Z}(\sigma_{\hbar,n})\big)g\|_{\ell^p(\hbar \Z^n)}^{p} &\leq &  \sum_{|k|>n_0}\Bigg(\Big(\varepsilon \big|\lambda\big|\ast\big| g\big|\Big)(k)\Bigg)^{p}\\
   &=& \varepsilon ^p\|\lambda\ast g\|_{\ell^{p}(\hbar \Z^n)}^{p}\\
   &\leq &  \varepsilon^{p}\|\lambda\|_{\ell^{1}(\hbar \Z^n)}^{p}\|g\|_{\ell^{p}(\hbar \Z^n)}^{p}.
   \end{eqnarray*}
   Finally by the density of $\ell ^1(\hbar \Z^n)$ in $\ell ^p(\hbar \Z^n)$ we obtain
   \[\|\text{Op}_{\hbar}(\sigma_\hbar) - \text{Op}_{\hbar}(\sigma_{\hbar,n})\|_{\mathscr{L}(\ell^p(\hbar \Z^n))} \leq  \varepsilon \|\lambda\|_{\ell^1(\hbar \Z^n)}\,,\]
  and the proof is complete.
   \end{proof}
  \section{Approximation of the classical Euclidean case}\label{Sec.h=0}
  In this section we recover known results on pseudo-differential operators in the Euclidean setting by allowing $\hbar \rightarrow 0$ in the semi-classical setting. Observe that whenever $\hbar \rightarrow 0$, the semi-classical setting $\hbar \mathbb{Z}^n$ ``approximates'' the Euclidean space $\mathbb{R}^n$.  We shall use the notation $x, \xi$ for elements in $\mathbb{R}^n$, while for elements in the semi-classical setting $\hbar \mathbb{Z}^n$ and for the toroidal elements we keep the same notation.
  
  To start our analysis, note that the definition of the  difference operators $\Delta_{\hbar,j}$ as in \eqref{EQ:diffs2}, when applied on function on $\mathbb{R}^n$ can be regarded as follows:  for fixed $x \in \mathbb{R}^n$ we have 
  \[
  \Delta_{\hbar,j}f(x):=\frac{f(x+e_j\hbar)-f(x)}{\hbar }\,,
  \]
  where $e_j=(0,\dots,0,1,0,\dots,0)$ is the vector with $1$ is at the $j^{th}$ position.  Then, for a function  $f: \mathbb{R}^n \rightarrow \mathbb{R}$ in the limiting case when $\hbar \rightarrow 0$ the difference  $\Delta_{\hbar,j}f(x)$ applied on $f$ ``approximates'' the corresponding classical partial derivative $\frac{\partial}{\partial x_j}f(x)$, or more generally we have the ``approximation'' for $x \in \mathbb{R}^n$:
  \begin{equation}\label{app1}    \Delta_{\hbar}^{\alpha}f(x)=\Delta_{\hbar,1}^{\alpha_1}\cdots \Delta_{\hbar,n}^{\alpha_n}f(x) \longrightarrow  \frac{\partial^{\alpha_1}}{\partial {x_1}^{\alpha_1}} \cdots \frac{\partial^{\alpha_n}}{\partial {x_n}^{\alpha_n}}f(x)=\partial_{x}^{\alpha}f(x)\,,\quad \alpha \in \mathbb{N}^n\,.  \end{equation}

  We note that in the expression \eqref{app1} we make abuse of the notation to describe the aforesaid notion of ``approximation''. 

  On the other hand, to ensure that the dual space of $\hbar \Zn$ ``approximates'' when $\hbar \rightarrow 0$ the dual space of $\mathbb{R}^n$ (which is $\mathbb{R}^n$ itself) we need to make the following change of variable: for $\theta \in \T^n$ we set $\omega=\frac{1}{\hbar}\theta \in \frac{\T^n}{\hbar}:=\mathbb{T}^{n}_{\hbar}$. It is then clear that in the limiting case the rescaled torus  $\T^{n}_{\hbar}$ ``approximates'' the Euclidean space $\mathbb{R}^n$. With this change of variable, the partial-type derivatives $D_{\hbar, \theta}^{(\beta)}$ introduced in Definition \ref{part.der.theta} become:
  \begin{eqnarray*}
      D_{\hbar, \theta}^{(\beta)} & = & D_{\hbar, \theta_1}^{(\beta_1)} \times \cdots \times D_{\hbar, \theta_n}^{(\beta_n)} \\
      & = & \hbar^{\beta_1} \left(\prod_{\ell=0}^{\beta_1-1}\frac{1}{2\pi i}\frac{\partial}{\partial \theta_1}-\ell\right) \times \cdots \times \hbar^{\beta_1}\left(\prod_{\ell=0}^{\beta_m-1}\frac{1}{2\pi i}\frac{\partial}{\partial \theta_n}-\ell\right)\\
      & = & \left(\prod_{\ell=0}^{\beta_1-1}\frac{1}{2\pi i}\frac{\partial}{\partial \omega_1}-\hbar \ell\right) \times \cdots \times\left(\prod_{\ell=0}^{\beta_m-1}\frac{1}{2\pi i}\frac{\partial}{\partial \omega_n}-\hbar \ell\right)\\   
      & = :  & d_{\hbar, \omega_1}^{(\beta_1)} \times \cdots \times d_{\hbar, \omega_n}^{(\beta_n)}  = d_{\hbar, \omega}^{(\beta)}\,,     \end{eqnarray*}
for some $\beta \in \mathbb{N}^n$, where $\omega=(\omega_1,\cdots, \omega_n) \in \T^{n}_{\hbar}$. Hence, when $\hbar \rightarrow 0$, we have the following ``approximation'' for $\xi \in \mathbb{R}^n$ and for a function $f: \mathbb{R}^n \rightarrow \mathbb{R}$:
\begin{equation}
    \label{app2}
    d_{\hbar, \omega}^{(\beta)}f(x) \longrightarrow \left(\frac{1}{2 \pi i} \right)^{|\beta|}\partial_{\xi}^{\beta}f(x)\,,
    \end{equation}
where $|\beta|=\beta_1+\cdots+\beta_n$ stands for the length of $\beta \in \mathbb{N}^n$.

Let us now discuss in which sense the semi-classical classes of symbols ``approximate'' in the limiting case the usual Euclidean H\"ormander classes of symbols introduced by H\"ormander in \cite{Hor66} and bearing his name. To begin with let us recap both definitions.
\medskip

\underline{\textbf{Semi-classical classes of symbols:}} Let $\rho, \delta, \mu \in \mathbb{R}$, and let $\sigma_{\hbar}: \hbar \Zn \times \Tn \rightarrow \mathbb{C}$. We say that $\sigma_{\hbar} \in S_{\rho,\delta}^\mu(\hbar\Z^n \times\T^n)$ if for all $\alpha,\beta \in \mathbb{N}_{0}^{n}$ there exists $C_{\alpha, \beta}$ such that 
    \begin{equation}\label{semi.symb.clas}
        |D_{\hbar, \theta}^{(\beta)} \Delta^\alpha_{\hbar,k} \sigma_{\hbar}(k,\theta)|\le C_{\alpha,\beta}(1+|k|)^{\mu-\rho|\alpha|+\delta|\beta|}\,.
        \end{equation}

        \medskip 

        \underline{\textbf{H\"ormander classes of symbols:}} Let $\mu \in \mathbb{R}$, $\delta<1$, $0 \leq \delta\leq \rho \leq 1$ and $\sigma : \mathbb{R}^n \times \mathbb{R}^n \rightarrow \mathbb{C}$. We say that $\sigma \in S^{\mu}_{\rho,\delta}(\mathbb{R}^n \times \mathbb{R}^n)$ if for all $\alpha,\beta \in \mathbb{N}_{0}^{n}$ there exists $C_{\alpha, \beta}$ such that 
\begin{equation}
    \label{hormander}
    |(\partial_{\xi}^{\alpha}\partial_{x}^{\beta}\sigma)(x,\xi)|\leq C_{\alpha,\beta} (1+|\xi|)^{\mu-\rho|\alpha|+\delta|\beta|}\,.
\end{equation}
\medskip 

Let us now restate the condition \eqref{semi.symb.clas} with respect to dual variable $\omega \in \T^{n}_{\hbar}$ involving the partial-type derivatives $d_{\hbar, \omega}^{(\beta)}$: 
\[
\sigma_{\hbar} \in S^{\mu}_{\rho,\delta}(\hbar \Zn \times \T^{n}_{\hbar}) \quad \text{if and only if} \quad |d_{\hbar, \omega}^{(\beta)}\Delta_{\hbar, k}^{(\alpha)}\sigma_{\hbar}(k, \omega)| \leq C_{\alpha, \beta}(1+|k|)^{\mu-\rho |\alpha|+\delta |\beta|}\,.
\]
When $\hbar$ approximates $0$ then the latter condition becomes:
\begin{equation}
    \label{classes.h0}
    \sigma_0 \in \tilde{S}^{\mu}_{\rho,\delta}(\mathbb{R}^n \times \mathbb{R}^n) \quad \text{if and only if} \quad |(\partial_{\xi}^{\beta}\partial_{x}^{\alpha}\sigma_0)(x,\xi)| \leq C_{\alpha, \beta}(1+|x|)^{\mu-\rho |\alpha|+\delta |\beta|}\,,\end{equation}
where, with an abuse of notation, we have assumed that $\sigma_{\hbar}(k, \omega) \xrightarrow[\hbar \to 0]{} \sigma_0(x, \xi)$. 

Observe that in the definition semi-classical classes of symbols \eqref{semi.symb.clas} the order of derivatives do not follow the lines of the classical H\"ormander classes of symbols as in \eqref{hormander}. This differentiation allows, after interchanging the role of $x$ and $\xi$ in \eqref{classes.h0}, for the following ``approximation'' of the above symbols classes in the two different settings provided that $0 \leq \delta \leq \rho \leq 1$:
\[
S^{\mu}_{\rho,\delta}(\hbar \Zn \times \T^{n}_{\hbar}) \xrightarrow[\hbar \to 0]{} \tilde{S}^{\mu}_{\rho,\delta}(\mathbb{R}^n \times \mathbb{R}^n)\,,
\]
where we use the notation $\tilde{S}^{\mu}_{\rho,\delta}(\mathbb{R}^n \times \mathbb{R}^n)$ for the associated symbol classes when the roles of $x$ and $\xi$ are interchanged; that is
\[
\sigma \in \tilde{S}^{\mu}_{\rho,\delta}(\mathbb{R}^n \times \mathbb{R}^n)\quad \text{if and only if}\quad |(\partial_{\xi}^{\beta}\partial_{x}^{\alpha}\sigma)(x,\xi)| \leq C_{\alpha, \beta}(1+|x|)^{\mu-\rho |\alpha|+\delta |\beta|}\,. 
\]
 With the above considerations we also have that $S^{-\infty}(\hbar \Zn \times \T^{n}_{\hbar})\xrightarrow[\hbar \to 0]{} \tilde{S}^{-\infty}(\mathbb{R}^n \times \mathbb{R}^n)$; that is, in the limiting case, the smoothing semi-classical pseudo-differential operators as in Definition \ref{smooth.semi} can considered to be negligible in the sense that when applied to distributions they produce rapidly decaying functions. 

 To give a meaning to the above ``aproximation'' of symbol classes in the two settings, let us state how the composition formula, see Theorem \ref{THM:comp}, the formulae for the adjoint and transpose of a $\Psi$DO in the semi-classical setting, see Theorem \ref{THM:adjoint} and Theorem \ref{THM:transpose}, respectively, and the asymptotic sum of $\Psi_\hbar$DOs, see Lemma \ref{LEM:as.sums}, become in the limiting case when $\hbar$ approaches $0$.  

 Below we discuss the above aspects of the symbolic calculus. With an abuse of notation we assume that for the  $\sigma_\hbar, \tau_\hbar$ semi-classical symbols in the hypothesis of the aforesaid theorems,  we have $\sigma_{\hbar}(k, \omega) \xrightarrow[\hbar \to 0]{} \sigma_0(x, \xi)$ and $\tau_{\hbar}(k, \omega) \xrightarrow[\hbar \to 0]{} \tau_0(x, \xi)$, where $(k, \omega) \in \hbar \Zn \times \T^{n}_{\hbar}$ and $x,\xi \in \mathbb{R}^n$.

 \medskip

 \underline{\textbf{Composition formula when $\hbar \rightarrow 0$:}} As noted in the discussion that follows after Theorem \ref{THM:comp} on the composition formula, the order of taking differences and derivatives in the corresponding asymptotic sum \eqref{EQ:comp} is different from the one in the Euclidean case. However, this differentiation allows to recover the classical composition formula for $\Psi$DO in the Euclidean setting. Indeed, Theorem \ref{THM:comp} when $\hbar \rightarrow 0$ identifies with the Euclidean one and in particular can be regarded as:  For $\sigma_0 \in \tilde{S}^{\mu_1}_{\rho, \delta}(\mathbb{R}^n \times \mathbb{R}^n)$ and $\tau_0 \in \tilde{S}^{\mu_2}_{\rho, \delta}(\mathbb{R}^n \times \mathbb{R}^n)$, the operator $\Op(\sigma_0) \circ \Op(\tau_0)$ has symbol $\varsigma_0 \in \tilde{S}^{\mu_1+\mu_2}_{\rho, \delta}(\mathbb{R}^n \times \mathbb{R}^n)$ given by the asymptotic sum 
 \[
 \varsigma_0(x,\xi) \sim  \sum_{\alpha} \frac{(2 \pi i)^{-|\alpha|}}{\alpha !}(\partial_{\xi}^{\alpha}\sigma_0)(x,\xi) (\partial_{x}^{\alpha} \tau_0)(x, \xi)\,, 
 \]
 where the factor $(2 \pi i)^{-|\alpha|}$ is due to the ``aproximation'' \eqref{app2} of partial-type derivatives, and is exactly the composition formula for $\Psi$DO with symbols in the H\"ormander classes in the Euclidean setting; see Theorem 2.5.1 in \cite{ruzhansky2009pseudo}.

 \medskip

 \underline{\textbf{Adjoint operator when $\hbar \rightarrow 0$:}} Before stating how the Theorem \ref{THM:adjoint} in the semi-classical setting reads in the limiting case, let us point out that, reasoning as above, the $\ell^2(\hbar \Zn)$-adjoint should be regarded as the $L^2(\mathbb{R}^n)$-adjoint as $\hbar$ approaches $0$.  In this sense, Theorem \ref{THM:adjoint} in the limiting case can be viewed as: For $\sigma_0 \in S^{\mu}_{\rho, \delta}(\mathbb{R}^n \times \mathbb{R}^n)$, there exists a symbol $\sigma_{0}^{*} \in S^{\mu}_{\rho, \delta}(\mathbb{R}^n \times \mathbb{R}^n)$ such that $\Op(\sigma_0)^{*}=\Op(\sigma_{0}^{*})$
, where $\Op(\sigma_0)^{*}$ is the $L^2(\mathbb{R}^n)$-adjoint of $\Op(\sigma_0)$, and we have the asymptotic expansion 
\[
\sigma_{0}^{*}(x,\xi) \sim \sum_{\alpha} \frac{(2 \pi i)^{-|\alpha|}}{\alpha !} \partial_{\xi}^{\alpha}\partial_{x}^{\alpha}\overline{\sigma_{0}(x, \xi)}\,,
\]
where as before the factor $(2 \pi i)^{-|\alpha|}$ is due to \eqref{app2} so that the above formula agrees with the Euclidean one; see Theorem 2.5.13 in \cite{ruzhansky2009pseudo}. 

\medskip 

\underline{\textbf{Transpose operator when $\hbar \rightarrow 0$:}} When $\hbar$ approaches $0$, Theorem \ref{THM:transpose} agrees with the corresponding result in the Euclidean setting, see Section 2.5 in \cite{ruzhansky2009pseudo}, and can be regarded as follows: For $\sigma_0 \in S^{\mu}_{\rho, \delta}(\mathbb{R}^n \times \mathbb{R}^n)$, there exists a symbol $\sigma_{0}^{t} \in S^{\mu}_{\rho, \delta}(\mathbb{R}^n \times \mathbb{R}^n)$ so that $\Op(\sigma_0)^{t}=\Op(\sigma_{0}^{t})$, and we have the following asymptotic expansion 
\[
\sigma_{0}^{t}(x,\xi) \sim \sum_{\alpha} \frac{(2 \pi i)^{-|\alpha|}}{\alpha !} \partial_{\xi}^{\alpha}\partial_{x}^{\alpha}[\sigma_{0}(x, -\xi)]\,.
\]

\medskip 

\underline{\textbf{Asymptotic sums when $\hbar \rightarrow 0$:}} When $\hbar$ approaches $0$, Lemma \ref{LEM:as.sums} can be viewed as follows: let $\big\{\mu_j\big\}_{j=0}^{\infty} \subset \mathbb{R}$ be a decreasing sequence such that  $\mu_j\rightarrow -\infty$ as $j\rightarrow \infty$, and let 
   If $\sigma_{0,j} \in S_{\rho, \delta}^{\mu_j}( \mathbb{R}^n\times\mathbb{R}^n)$  for all $j\in \mathbb{N}_0$.  Then, there exists  $\sigma_0 \in S_{\rho, \delta}^{\mu_0}(\mathbb{R}^n\times\mathbb{R}^n)$ such that 
   \[\sigma_{\hbar} \sim \sum_{j=0}^{\infty}\sigma_{\hbar,j},\]
   which means that for all $N \in \mathbb{N}$ we have 
   \[\sigma_{0} - \sum_{j=0}^{N-1}\sigma_{0,j} \in S_{\rho, \delta}^{\mu_N}(\mathbb{R}^n\times\mathbb{R}^n)\,. \]
   The latter agrees with Proposition 2.5.33 in \cite{ruzhansky2009pseudo}. 

   \medskip

   In the last part of the current subsection we discuss the $\ell^2(\hbar \Zn)$-boundedness of semi-classical operators in the limiting case when $\hbar $ approaches zero where the latter statement translates into $L^2(\mathbb{R}^n)$-boundedness. 
   
   Let us first recall a celebrated result on the boundedness of pseudo-differential operators in the Euclidean setting, see e.g.  the monograph of Stein \cite{St93}.  
   \begin{thm}\label{Stein}
   If $\sigma \in S^{0}_{\rho,\delta}(\mathbb{R}^n \times  \mathbb{R}^n)$, then $\Op_{\mathbb{R}^n}(\sigma)$ is a bounded operator from $L^2(\mathbb{R}^n)$ to $L^2(\mathbb{R}^n)$.   
   \end{thm}

   \underline{\textbf{On the boundedness of $\Psi_\hbar$DO when $\hbar \rightarrow 0$:}} 
Let us first see how Theorem \ref{THM:L2} translates in the limiting case: Let $\kappa \in \mathbb{N}$, $\kappa > n/2$, and let  $\sigma_0 : \mathbb{R}^n \times \mathbb{R}^n \rightarrow \mathbb{C}$ be a symbol satisfying 
\begin{equation}\label{cont.cond.lim}
|\partial_{\xi}^{\alpha}\sigma_0(x,\xi)| \leq C\,,\quad \text{for all} \ x,\xi \in \mathbb{R}^n\,,
\end{equation}
where $|\alpha| \leq \kappa$. Then, $\Op_{\mathbb{R}^n}(\sigma_0)$ extends to a bounded operator on $L^2(\mathbb{R}^n)$. Interestingly, going back to the analogous result on the Euclidean setting,  we see that Theorem \ref{Stein} implies the limiting condition \eqref{cont.cond.lim} for all $\alpha \in \mathbb{N}^n$. Indeed, let $\sigma_0 \in S^{0}_{\rho,\delta}(\mathbb{R}^n \times  \mathbb{R}^n)$. Then, by the assumption on $\sigma_0$ and inequality \eqref{hormander}, we have 
\[
|\partial_{\xi}^{\alpha}\sigma_0(x,\xi)| \leq C (1+|\xi|)^{-\rho |\alpha|} \leq C\,,\quad \text{for all} \ \alpha \in \mathbb{N}^n\,.
\]
Conversely, condition \eqref{cont.cond.lim} implies the assumption of Theorem \ref{Stein} provided that $\delta |\beta| \geq \rho |\alpha|$. 
  \section{Examples}
  \label{SEC:EX}
  
 In this last section we consider certain examples of difference equations where semi-classical pseudo-differential operators are involved. In particular, in the subsequent examples, making use of the analysis above we study: the order of the corresponding semi-classical symbol, the boundedness of the operator, the ellipticity of the symbol and the existence, or even the exact formula, of the parametrix.   
 
 In what follows we  denote by $v_j$ the unit vector $v_{j} = (0, \ldots, 0, 1, 0, \ldots, 0)\in \mathbb{Z}^n$, where $1$ is the $j^{th}$ entry of the vector.  
    \begin{ex}
    Below we list several cases of semi-classical difference equations:
    \begin{enumerate}
        \item (Case of non-elliptic bounded operator of zero order) Let us define the operator $D_j$ by
          \[D_{j}f(k) = f(k+\hbar v_j) - f(k)\,,\quad \text{where}\quad k \in \hbar \Zn\,.\]
     For $\theta \in \Tn$, let $e_\theta: \hbar \Zn \rightarrow \mathbb{R}$ be the function given by $e_\theta(k)=: e^{2\pi \frac{i}{\hbar} k\cdot \theta}\,.$ Then since 
    \[D_{j}e_{\theta}(k) = e^{2\pi \frac{i}{\hbar} (k + \hbar v_j)\cdot \theta} - e^{2\pi \frac{i}{\hbar} k\cdot \theta},\] 
    by Proposition \ref{PROP:symbols} the symbol of $D_j$ is given by
    \[\sigma_{\hbar,D_{j}}(k,\theta) = e^{2\pi i v_j\cdot \theta} - 1 = e^{2\pi i \theta_j} - 1.\]
   Thus, the operator $D_j$ is not elliptic, we have $\Op(\sigma_{\hbar,D_j}) \in \Op(S^{0}(\hbar \Zn \times \Tn))$ and is bounded by Corollary \ref{COR:L2} from $\ell^2(\hbar \Z)$ to $\ell^2(\hbar \Z)$.
   \item (Case of elliptic bounded operator of positive order)  Let us define the operator
   \[L_jf(k) = |k|^{r}(f(k+\hbar v_j)+a)-|k|^s (f(k-\hbar v_j)+b)\,,\]
for some $a,b \in \mathbb{R}$ such that $|a|, |\beta|\geq 1$.   Its symbol is given by 
    \[\sigma_{\hbar,L_{j}}(k,\theta) = |k|^r(e^{2\pi i \theta_j}+e^{-2\pi \frac{i}{\hbar}k \cdot\theta}a)-|k|^{s}(e^{-2\pi i \theta_j}+e^{-2\pi \frac{i}{\hbar}k \cdot\theta}b)\,.\]
    We have $\sigma_{\hbar,L_{j}}(k,\theta) \in  S^{\max\{s,r\}}(\hbar \mathbb{Z}^n \times \mathbb{T}^n)$.
    This symbol is  elliptic of order $r$ if $r\geq s$, and non-elliptic otherwise.
    Consequently, from  Corollary \ref{COR:L2} and for $r \geq s$, the operator $L_j$ is bounded from the weighted space $ \ell^2_{t+s}(\hbar \Zn)$ to the weighted space $ \ell^2_{t}(\hbar \Zn)$, for any $t \in \mathbb{R}$.
  \item   (Case of elliptic bounded operator of zero order)  Let us define the operator
     \[{T}f(k) := \sum_{j=1}^{n}\Big(f(k + \hbar v_j) + f(k - \hbar v_j)\Big) + cf(k)\,,\quad \text{where}\quad c \in \mathbb{C}\,.\]
     The explicit formula of its symbol can be found as follows	
    	\[
     \sigma_{\hbar,T}(k,\theta)=\sum_{i=1}^{n} \left(e^{2\pi i \theta_j}-e^{-2\pi i \theta_j} \right)+c=2i \sum_{j=1}^{n}\sin(2 \pi \theta_j)+c \in S^{0}(\hbar \Zn \times \Tn)\,.
        \]
        and is elliptic in the case where $\textnormal{Re}\, c \neq 0$ or $\textnormal{Im}\, c  \notin [-2n,2n]$. Under such assumptions on $c$ the inverse operator $T^{-1}$ is also of order zero, and its symbol that depends only on the toroidal variable $\theta$ is given by 
        \[
        \sigma_{\hbar,T^{-1}}(\theta)=\frac{1}{2i \sum_{j=1}^{n}\sin(2 \pi \theta_j)+c}\,.
        \]
        Hence the solution to the equation
        \begin{equation}\label{ex.eq}
        Tf(k)=g(k)\,,
        \end{equation}
        is given by
        \[T^{-1}g(k)=\Op(\sigma_{\hbar, T^{-1}}g(k)=\int_{\Tn}e^{2\pi \frac{i}{\hbar} k\cdot \theta}\frac{1}{2i \sum_{j=1}^{n}\sin(2 \pi \theta_j)+c} \widehat{g}(\theta)\dd \theta\,. \]
        By Corollary \ref{COR:L2} the operators $T,T^{-1}$ are bounded from $\ell^{2}_{s}(\hbar \Zn)$ to $\ell^{2}_{s}(\hbar \Zn)$, which implies that if $g \in \ell^{2}_{s}(\hbar \Zn)$ then for the solution to the equation \eqref{ex.eq} we also have $f \in \ell^{2}_{s}(\hbar \Zn)$.   
    \end{enumerate}
    \end{ex}


 \end{document}